\documentclass[11pt]{article}
\usepackage{geometry}
\usepackage{amsfonts}
\usepackage{amssymb}
\usepackage{graphicx}
\usepackage{amsmath}
\usepackage{amsthm}
\usepackage{enumitem}
\usepackage{tikz}
\usepackage{mathrsfs}
\usepackage{cancel}
\usepackage{todonotes}
\usepackage{cases}
\usetikzlibrary{matrix}
\usetikzlibrary{arrows,shapes}
\usepackage{booktabs}
\usepackage{cite}
\usepackage{bm}
\usepackage{hyperref}
\hypersetup{colorlinks=true, urlcolor=blue}
\expandafter\let\expandafter\oldproof\csname\string\proof\endcsname
\let\oldendproof\endproof
\renewenvironment{proof}[1][\proofname]{%
  \oldproof[\ttfamily \scshape \bf #1. ]%
}{\oldendproof}

\usepackage{float}
\usepackage{algorithm}
\usepackage{algpseudocode}

\def\A{{\cal A}}

\def\s{{\mathcal S}}
\def\B{\mathbb{B}}
\def\R{{\mathbb R}}
\def\oR{\overline{\R}}
\def\N{{\rm I\!N}}

\def\ox{{\bar{x}}}
\def\tx{\tilde{x}}
\def\ty{\tilde{y}}
\def\oy{\bar{y}}
\def\oz{\bar{z}}
\def\ov{\bar{v}}

\def\bu{{\bf u}}
\def\bv{{\bf v}}
\def\X{{\R^n}}
\def\Y{{\R^m}}

\def\H{{\bf H}}

\def \b{{\}_{k\in\N}}}

\def\tto{\rightrightarrows}

\def\Lm{{\Lambda}}
\def\tto{\rightrightarrows}

\def\sub{\partial}

\def\ra{\rangle}
\def\la{\langle}
\def\ve{\varepsilon}

\def\dd{\delta}

\def\ph{\varphi}

\def\toset_#1{\xrightarrow{#1}}
\DeclareMathOperator*{\mini}{minimize\;}
\DeclareMathOperator*{\argmin}{arg\, min}
\def\d{{\rm d}}
\def\dist{{\rm dist}}
\def\spann{{\rm span}\,}
\def\rge{{\rm rge\,}}

\def\cone{{\rm cone}\,}
\def\ri{{\rm ri}\,}

\def\gph{{\rm gph}\,}
\def\epi{{\rm epi}\,}

\def\dom{{\rm dom}\,}

\def\ker{{\rm ker}\,}

\def\sm{\hbox{${1\over 2}$}}

\def\K{{\overline{  K}}}

\def\x1k{{x^{k-1}}}
\def\xk{{x^k}}
\def\xkk{{x^{k+1}}}
\def\yk{y^k}
\def\ykk{y^{k+1}}
\def\zk{z^k}
\def\wk{w^k}
\def\uk{u^k}

\def\l1k{{\lambda^{k-1}}}

\def\sig1k{{\sigma_{k-1}}}
\def\sigk{{\sigma_k}}

\def\ro1k{{\rho_{k-1}}}

\def\e1k{{\epsilon_{k-1}}}

\def\Hk{{H_k}}

\def\ep3s{{\varepsilon_3}}
\def\el3l{{\ell_3}}
\def\hyk{\hat y^k}
\def\Ci{{C_{i_0}}}
\def\ii{{i_0}}

\begin{document}
\begin{center}
{\Large \bf Stabilized SQP Method for Convex Piecewise Linear-Quadratic Composite Optimization Problems}\\[2ex]
NGUYEN T. V. HANG\footnote{Faculty of Mathematics, University of Duisburg-Essen, 45127 Essen, Germany, and Institute for Artificial Intelligence, VNU University of Engineering and Technology, Hanoi, Vietnam (ntvhang@vnu.edu.vn). Research of this author is partially supported by the Alexander von Humboldt Foundation.} 
 and   EBRAHIM SARABI\footnote{Department of Mathematics, Miami University, Oxford, OH 45065, USA (sarabim@miamioh.edu). Research of this    author is partially supported by the U.S. National Science Foundation  under the grant DMS 2108546.
 }
\end{center}
\vspace*{0.05in}

\small{\bf Abstract.} 
This paper addresses the local convergence analysis of the stabilized sequential quadratic programming method for a class of composite optimization problems with convex piecewise linear-quadratic (CPLQ) modeling functions. 
We demonstrate that an appropriate second-order sufficient condition (SOSC)---formulated for composite problems---alone is sufficient to guarantee local superlinear convergence of the primal-dual sequence generated by this method. 
Furthermore, we examine the relaxation of the SOSC and show that it can be replaced with a weaker noncriticality of the associated Lagrange multiplier, provided, the modeling function is polyhedral and a relative interior condition is satisfied. 
We also present counterexamples showing that the relative interior condition is indispensable and that this relaxation fails for general CPLQ modeling functions.
 \\[1ex]

\newtheorem{Theorem}{Theorem}[section]
\newtheorem{Proposition}[Theorem]{Proposition}
\newtheorem{Lemma}[Theorem]{Lemma}
\newtheorem{Corollary}[Theorem]{Corollary}

\theoremstyle{definition}
\newtheorem{Definition}[Theorem]{Definition}
\newtheorem{Algorithm}[Theorem]{Algorithm}
\newtheorem{Example}[Theorem]{Example}
\newtheorem{Remark}[Theorem]{Remark}

\numberwithin{equation}{section}

\renewcommand{\thefootnote}{\fnsymbol{footnote}}

\normalsize

\section{Introduction} \label{sect:Intro}
This paper aims to investigate the local convergence of a stabilized version of the classical sequential quadratic programming (SQP) method for an important class of the composite optimization problem
\begin{equation}\label{comp}
    \mini_{x\in \X} \psi(x):=\ph(x)+ g(\Phi(x)),
\end{equation} 
where  $\ph:\X\to \R$ and $\Phi:\X\to \Y$ are twice continuously differentiable and the modeling function $g: \Y \to \oR:=[-\infty,\infty]$ is   convex piecewise linear-quadratic (CPLQ).
Recall from \cite[Definition~10.26]{rw} that a  function $g:\R^m\to \oR$ is piecewise linear-quadratic if  its domain $\dom g$ is the union of finitely many polyhedral convex sets $C_i$, indexed by $i\in I$, and on each polyhedral piece $C_i$,  the function $g$ has a quadratic representation in the form
\begin{equation} \label{PWLQ}
g(z) = \sm \langle A_i z ,z \rangle + \langle a_i ,z \rangle + \alpha_i  \quad \mbox{for all} \quad  z \in C_i,
\end{equation}
where $A_i$ is an $m \times m$ symmetric matrix, $a_i\in \R^m$, and $\alpha_i\in \R$ for all $i \in I$.  Due to their strong computational properties such as preservation under addition, linear composition (cf. \cite[Exercise~10.22]{rw}), and the Legendre-Fenchel transform (cf. \cite[Theorem~11.14]{rw}), CPLQ functions are widely used to model optimization problems across  machine learning, sparse/robust statistical inference, and stochastic optimization (see, e.g., \cite{ABP13, QiW95, RoW86}). 

A fundamental subclass of CPLQ functions consists of polyhedral functions---functions whose epigraphs are convex polyhedral sets.
This class encompasses popular regularization terms such as the Lasso, which plays a key role in least-squares regression in machine learning. 
Furthermore, by choosing $g$ as the pointwise maximum function, the composite problem \eqref{comp} covers minimax optimization problems. 
Note also that the classical nonlinear programs (NLPs) with $m_1$ inequality and $m_2$ equality constraints can be formulated via \eqref{comp} by setting $g$ as the indicator function $\dd_{\R^{m_1}_-\times \{0\}^{m_2}}$. 
Another important instance of CPLQ functions is the quadratic penalty function 
\begin{equation}\label{pen}
    g(z) :=\theta_{Y,Q}(z)= \sup_{w\in Y}\big\{\la z, w\ra - \tfrac{1}{2}\la Qw, w\ra\big\}, \quad z\in \Y,
\end{equation} 
where  $Y\subset \Y$ is a convex polyhedral set and $Q \in \R^{m\times m}$ is a  positive semidefinite matrix  (cf.~ \cite[Example~11.18]{rw}).
This penalty function appeared first in \cite{r87} in the context of extended linear-quadratic optimization problems.

In this paper, we mainly focus on the local convergence analysis of 
a stabilized version of the SQP method for  \eqref{comp}. Recall that 
iterations of the classical SQP method for the composite problem \eqref{comp} consist of solving the piecewise linear-quadratic subproblems 
\begin{align}\label{SQP}
\mini_{x\in \X}\, \Big\{ \la\nabla \varphi(\xk),x-\xk\ra + \tfrac{1}{2}\la\nabla^2_{xx}L(\xk, \yk)(x-\xk), x-\xk\ra\nonumber\;&\\
+\ g(\Phi(\xk) + \nabla\Phi(\xk)(x-\xk))\Big\},&
\end{align}
where $L(x, y) = \ph(x) + \la y, \Phi(x)\ra$, for $(x, y) \in \X\times\Y$, is the Lagrangian associated with \eqref{comp}, and $(\xk, \yk) \in \X\times\Y$ is the current primal-dual iterate. The new dual variable $\ykk$ then can be chosen as a Lagrange multiplier associated with a local minimizer $\xkk$ of \eqref{SQP}.
Another way to update the primal-dual iterate $(\xk, \yk)$ is to consider the next iterate $(\xkk, \ykk)$ as a solution to  the KKT system of \eqref{SQP}, which can be formulated as a variational system 
\begin{equation*}\label{SQP-KKT}
    \begin{cases}
        0= \nabla\ph(\xk) + \nabla \Phi(\xk)^*y + \nabla^2_{xx}L(\xk, \yk)(x-\xk),\\
        y\in \sub g\big(\Phi(\xk) + \nabla\Phi(\xk)(x-\xk)\big). 
    \end{cases}
\end{equation*}
The local convergence of the SQP method for \eqref{comp} was first studied in \cite[Theorem~9]{be} under a strong version of the second-order sufficient condition (SOSC) in combination with the nondegeneracy and strict complementary conditions. 
This was indeed an extension of a similar result for NLPs, established by  Robinson in \cite{rob74,rob80},  showing that the strong SOSC for NLPs, combined with the linear independence constraint qualification, ensures local superlinear/quadratic convergence of the SQP primal-dual sequence. 
This observation was substantially improved  in \cite{bo94}, where it was proven for the first time that those assumptions can be replaced with the classical SOSC and the uniqueness of Lagrange multipliers. 

Following the approach in \cite{bo94}, the local convergence of the SQP in \cite{be} was improved in \cite[Theorem~5.4]{s20} for the composite problem \eqref{comp}.
Specifically, the local superlinear/quadratic convergence of the SQP primal-dual sequence $\{(\xk,\yk)\b$, generated according to \eqref{SQP}, was established under a suitable adaptation of the classical SOSC for NLPs to the problem \eqref{comp} (see \eqref{sosc}), together with the dual condition 
\begin{equation}\label{dual}
    K_ g \big(\Phi(\bar x),\oy \big)^*\cap \ker \nabla \Phi(\ox)^* =\{0\},
\end{equation}
at the KKT point $(\ox, \oy)$ in question,
where $K_ g \big(\Phi(\bar x),\oy \big)$ denotes the critical cone of $g$ defined in \eqref{domcr}. 
Note that the condition \eqref{dual} amounts to the uniqueness of Lagrange multipliers associated with $\ox$ (cf. \cite[Proposition~3.10]{s20}).
This raises the fundamental question of whether local convergence analysis of the SQP method can be established for problems with nonunique Lagrange multipliers---hereafter referred to as degenerate problems. 
The stablized SQP (sSQP) method was introduced by Wright in \cite{w98, w02} for NLPs as a way to study the local convergence of SQP methods for degenerate NLPs. Given the current iterate $(\xk,\yk)$,  the subproblems of the sSQP method for NLPs with equality constraints, obtained from \eqref{comp} with $g=\dd_{\{0\}^m}$, 
can be formulated as 
\begin{equation}\label{ssqp-nlp}
 \begin{cases}
\mini\limits_{(x,y)\in \X\times \R^m}\,  \la\nabla \varphi(\xk),x-\xk\ra + \tfrac{1}{2}\la\nabla^2_{xx}L(\xk, \yk)(x-\xk), x-\xk\ra+\frac{\sigma_k}{2}\|y\|^2 \\
\mbox{subject to}\quad \Phi(\xk) + \nabla\Phi(\xk)(x-\xk)-\sigma_k(y-\yk)=0,
\end{cases}   
\end{equation}
where $\sigma_k>0$ is called the dual stabilization parameter. 
A major practical challenge in implementing the standard SQP method is the potential infeasibility of its subproblems as discussed in  \cite{fd}. 
In contrast, the sQP method incorporates elastic terms into the linearized constraints of its subproblems, thereby ensuring their feasibility without requiring any constraint qualification. 
In \cite{w98}, Wright demonstrated  that the local superlinear convergence of the sSQP primal-dual iterates can be ensured under the classical SOSC, together with the  Mangasarian-Fromovitz constraint qualification (MFCQ) and the strict complementarity condition. 
Note that the MFCQ corresponds to the boundedness of the set of Lagrange multipliers, thus accommodating NLPs with nonunique multipliers. 
Hager \cite{ha98} subsequently improved the result by  showing that both   MFCQ and   strict complementarity condition could be dropped if a stronger variant of the SOSC was adopted; see also \cite{Fis02} for a different approach that leads to a similar conclusion.
Whether a comparable convergence result could be achieved under the classical SOSC alone, without any constraint qualification, remained an open question until \cite{FeS10} proved that the classical SOSC is indeed sufficient to ensure local superlinear convergence of the sSQP primal-dual iterates. 

Recently, efforts have been made to extend the local convergence analysis of the sSQP method to broader classes of constrained optimization problems, such as semidefinite programming problems \cite{ya,zc}. 
However, these existing results all require the uniqueness of Lagrange multipliers, thereby leaving the local convergence analysis of sSQP for degenerate problems largely unaddressed. 
To fill this gap, our main objective is to demonstrate that second-order variational analysis provides effective tools for establishing a counterpart to the main result in \cite{FeS10} for the composite problem \eqref{comp}. 
As highlighted by Hager in \cite[page~254]{ha98}, the sSQP algorithm in \cite{w98} for NLPs is closely related to the augmented Lagrangian method (ALM; see, e.g., \cite{HaS21}) applied to the standard SQP subproblems \eqref{SQP}. 
Guided by this observation, we extend the sSQP framework to the composite optimization problem \eqref{comp}. 
That is, given the current iterate $(\xk, \yk)\in \X\times\Y$ and a dual stabilization parameter $\sigma_k>0$, the new sSQP iterate $(\xkk,\ykk)$ is generated by
\begin{equation}\label{sSQP}
\begin{cases}
    \xkk \in \argmin\limits_{x\in \X}\;  \Big\{\la\nabla \ph(\xk),x-\xk\ra  + \tfrac{1}{2}\la\nabla^2_{xx}L(\xk, \yk)(x-\xk), x-\xk\ra\\
    \hspace{1.8in} + \,(e_{\sigma_k}g)\big(\Phi(\xk) + \nabla\Phi(\xk)(x-\xk) + \sigma_k\yk\big)\Big\},\\
     \ykk = \nabla (e_{\sigma_k}g)\big(\Phi(\xk) + \nabla\Phi(\xk)(x-\xk) + \sigma_k\yk\big),
    \end{cases}
\end{equation}
where $e_\sigma g$ stands for the Moreau envelope of $g$ for a parameter $\sigma>0$, defined by 
\begin{equation*}\label{Menv}
(e_\sigma g)(z):=\inf_{w\in \Y}\big\{g(w)+\tfrac{1}{2}\sigma^{-1}\|w-z\|^2\big\}, \quad z\in \Y.
\end{equation*}
We emphasize that this connection between the sSQP method and ALM plays a pivotal role in establishing the key error bound estimates in our analysis.
Specifically, we utilize an approach analogous to our recent work \cite{HaS21} on the local convergence analysis of the ALM to justify Proposition~\ref{prop:eb}, which is an important tool toward proving the solvability of the sSQP subproblems. 

According to the characterization of the mapping $\nabla (e_\sigma g)$ (see \eqref{prm}), the primal-dual update $(\xkk,\ykk)$ in \eqref{sSQP} can be interpreted as finding a solution $(x, y)$ to the variational system
\begin{equation}\label{sSQP-KKT}
    \begin{cases}
        0= \nabla\ph(\xk)  + \nabla^2_{xx}L(\xk, \yk)(x-\xk) + \nabla \Phi(\xk)^*y,\\
        y \in \sub g\big(\Phi(\xk) + \nabla\Phi(\xk)(x-\xk) - \sigma_k(y - \yk)\big). 
    \end{cases}
\end{equation}
Indeed, this  is the KKT system of the smooth optimization problem in \eqref{sSQP}. 
Alternatively, following the idea in \eqref{ssqp-nlp}, we can formulate the sSQP subproblem for \eqref{comp} as
\begin{align}\label{sSQP-2}
    \mini_{(x, y) \in \X\times \Y} &\;\la\nabla \ph(\xk),x-\xk\ra + \tfrac{1}{2}\la\nabla^2_{xx}L(\xk, \yk)(x - \xk), x - \xk\ra\nonumber\\
    &\; +g\big(\Phi(\xk) + \nabla\Phi(\xk)(x - \xk) - \sigma_k(y - \yk)\big) + \tfrac{1}{2}\sigma _k\|y\|^2.
\end{align}
It is worth noting that any primal-dual solution $(\xkk, \ykk)$ to \eqref{sSQP-2}, if exists, satisfies the system of first-order optimality conditions for \eqref{sSQP} formulated in \eqref{sSQP-KKT}.

Our main contributions are twofold. First, we show, through an auxiliary variational system and the reduction lemma for CPLQ functions, that \eqref{sSQP-KKT} is always solvable under the SOSC. 
Combining this existence result with a new error bound estimate for consecutive sSQP iterates, we prove the local superlinear convergence of the sSQP primal-dual sequence. 
Second, we investigate the relaxation of the SOSC to the noncriticality of  Lagrange multipliers, demonstrating that such a weakening is possible when $g$ in \eqref{comp} is a polyhedral function and the Lagrangian multiplier satisfies a relative interior condition. 
The remainder of the paper is organized as follows. 
Section~\ref{sect:pre} introduces necessary notation, key variational properties of CPLQ functions, and an algorithmic framework from \cite{Fis02} for establishing superlinear convergence. 
Section~\ref{sect:sosc} analyzes the solvability of the variational system \eqref{sSQP-KKT} under the SOSC, while Section~\ref{sec:noncri} addresses the same issue under the noncriticality of Lagrange multipliers for polyhedral modeling functions. 
We also present counterexamples to rule out the possibility of such a result for the CPLQ case. 
Finally, Section~\ref{sec:convs} establishes our main local convergence results for the sSQP method under both conditions.

\section{Preliminaries}\label{sect:pre}
\subsection{Notation}
Throughout the paper, $\B$ denotes the closed unit ball in the space under consideration and $\B_r(x):=x+r\B$ represents the closed ball centered at $x$ with radius $r>0$.
The product space $\X\times \Y$ is endowed with the norm $\|(u,w)\|=\|u\|+\|w\|$.
For a matrix $A \in \R^{m\times n}$, $A^*$ stands for its transpose.
For a nonempty set $C\subset\X$, we denote by $\ri C$, $C^*$, $\spann C$, and $\cone C$ its relative interior (the interior relative to the affine hull of $C$), polar cone, linear span, and the closed conical hull, respectively. 
The indicator function of $C$ is defined by $\dd_C(x)=0$ for $x\in C$ and $\dd_C(x)=\infty$ otherwise, while $P_C$ and $\dist(\cdot,C)$ denote the Euclidean projection onto $C$ and the distance function to $C$.
For a nonzero vector $w\in \X$, $[w] := \{tw~|~t\in \R\}$ denotes the line spanned by $w$, and $[w]^\perp$ is its orthogonal complement. 
The sets of nonnegative and nonpositive numbers are denoted by $\R_+$ and $\R_-$, respectively.

For a nonempty set $\Omega\subset\X$ and $\ox\in \Omega$, the tangent cone  to $\Omega$ at $\ox$, denoted $T_\Omega(\ox)$,  is defined  by
\[
T_\Omega(\ox) = \big\{u\in \X~\big|~ \exists t_k\searrow0,\, u^k \to u\, \text{ with }\, \ox + t_k u^k \in \Omega\big\}.
\]
If $\Omega$ is convex, then the normal cone to $\Omega$ at $\ox$ is given by 
\[
N_\Omega(\ox) = T_\Omega(\ox)^* = \big\{ v\in \X~\big|~\la v, x- \ox\ra \leq 0 \, \text{ for all } x\in \Omega\big\}.
\]

For a function $f:\X \to \oR$, its domain and epigraph are defined, respectively, by
\[
\dom f = \big\{x\in \X~\big|~ f(x)<\infty\big\}\quad\text{ and } \quad
\epi f = \big\{(x, \alpha)\in \X \times \R~\big|~ f(x)\leq \alpha\big\}.
\]
Assuming that $f$ is twice differentiable at $\ox$ with gradient $\nabla f(\ox)$ and Hessian $\nabla^2f(\ox)$, we use the same notation $\nabla^2f(\ox)$ to denote the self-adjoint linear operator of $\X$ associated with the bilinear symmetric form $\nabla^2f(\ox):\X\times\X\to \R$.
Fix a point $\ox\in\X$ with $f(\ox)$ finite, the subderivative function $\d f(\ox)\colon\X\to\oR$ is defined by
\begin{equation*}\label{fsud}
\d f(\ox)(w)=\liminf_{\substack{
   t\searrow 0 \\
  w'\to w
  }} \, \frac{f(\ox+tw')-f(\ox)}{t}.
\end{equation*} 
 
 Assume that $f$ is convex, its subdifferential at $\ox$ is given by
\[
\sub f(\ox) = \big\{ v\in \X~\big|~ f(x) \geq f(\ox) + \la v, x- \ox\ra \, \text{ for all } x\in \X\big\}.
\]
The critical cone of $f$ at $\ox$ for $\bar v$ with $\bar v\in   \sub f(\ox)$ is defined by 
\begin{equation*}\label{cricone}
{K_f}(\ox,\bar v)=\big\{w\in \X\,\big|\,\la\bar v,w\ra=\d f(\ox)(w)\big\}.
\end{equation*}
When $f=\dd_\Omega$, where $\Omega$ is a nonempty subset of $\X$, the critical cone of $\dd_\Omega$ at $\ox$ for $\ov$ is denoted by $K_\Omega(\ox,\ov)$. In this case, the above definition of the critical cone of a function 
boils down to  the well-known concept of the critical cone of a set (see \cite[page~109]{dr}), namely $K_\Omega(\ox,\ov)=T_\Omega(\ox)\cap [\ov]^\perp$. 
The {second subderivative} of $f$ at $\ox$ for $\ov \in   \sub f(\ox)$, denoted $\d^2 f(\bar x , \ov)$, is an extended-real-valued function defined  by 
\begin{equation*}\label{ssd}
\d^2 f(\bar x , \ov)(w)= \liminf_{\substack{
   t\searrow 0 \\
  w'\to w
  }} \frac{f(\ox+tw')-f(\ox)-t\langle \ov,\,w'\rangle}{\frac {1}{2}t^2},\;\; w\in \X.
\end{equation*}
Recall that  the proximal mapping of  a convex function $f:\X\to \oR$ is defined by 
\begin{equation}\label{pr}
{\rm{prox}}_{\sigma f}(x)=\argmin_{u\in \X} \big\{\sigma f(u)+\tfrac{1}{2}\|u-x\|^2\big\},\quad x\in\X,
\end{equation}
where $\sigma$ is a positive constant. 
In what follows, when $\sigma = 1$, the proximal mapping of $f$ will be denoted by ${\rm prox}_f$. 
It follows from \cite[Theorem~2.26]{rw} and optimality condition for \eqref{pr} that for any $x\in \X$ we always have 
\begin{equation}\label{prm}
y= \nabla (e_\sigma f)(x)=\sigma^{-1}\big(x - {\rm{prox}}_{\sigma f}(x)\big) \iff
y \in \sub f(x - \sigma y).
\end{equation}

For  a set-valued mapping $F: \X \rightrightarrows \Y$, its graph and domain are defined, respectively,  by  
\[
\gph F=\big\{(x,y)\in \X\times \Y~\big|~y\in F(x)\big\}\quad\text{ and }\quad
\dom F= \big\{ x\in \X~\big|~ F(x)\neq \emptyset\big\}.
\]
According to \cite[Definition~8.33]{rw}, the {graphical derivative} of $F$ at $\ox$ for $\oy \in F(\ox)$ is the set-valued mapping $DF(\ox, \oy): \X\tto \Y$ defined by
\begin{equation*}\label{gder}
\eta \in DF(\ox, \oy)(w)\; \Longleftrightarrow\; (w,\eta)\in T_{\gph F}(\ox, \oy),
\end{equation*}
or, equivalently, $\gph DF(\ox, \oy) = T_{\gph F}(\ox, \oy)$.

\subsection{Variational Properties of CPLQ Functions}  \label{sect:CPLQ}

Let the model function $g:\R^m\to \oR$ be CPLQ with the representation \eqref{PWLQ}. Take $\oz\in \dom g = \cup_{i\in I}C_i$ and define the {active} indices  of the domain of $g$ at $\oz$ by
$I(\oz)=\big \{i\in I ~\big|~ \oz\in C_i\big \}$. 
It follows   from \cite[page~487]{rw} that the subdifferential of $g$ at  $\oz$ can be calculated by 
\begin{equation*}\label{sub}
\sub g(\oz)=\bigcap_{i\in I(\oz)} \big \{y\in \R^m~\big|~y-A_i \oz- a_i\in N_{ C_i}(\oz)\big \}.
\end{equation*}
Fix $\oy \in \sub g(\oz)$ and set $\oy_i:=\oy -A_i \oz -a_i$ for all $i\in I(\oz)$. 
The critical cone of $g$ at $\oz$ for $\oy$ is then given by
\begin{equation}\label{domcr}
K_g(\oz,\oy) = N_{\sub g(\oz)}(\oy) =\bigcup_{i\in I(\oz)} {K}_{C_i}(\oz,\oy_i).
\end{equation}
According to to \cite[Proposition~13.9]{rw}, the second subderivative of $g$ at $\oz$ for $\oy$, calculated for any $w\in \R^m$ by
\begin{equation*}\label{pwfor}
 \d^2 g(\oz , \oy ) (w)  
=\begin{cases}
\la A_i w, w\ra&\mbox{if}\;\; w\in {K}_{C_i}(\oz,\oy_i),\\
\infty&\mbox{otherwise}.
\end{cases}
\end{equation*} 
Let $B$ be an $m\times n$ matrix and 
$h(u):= \d^2 g(\oz , \oy ) (Bu)$, for $u\in \R^n$.
It follows from the chain rule in \cite[Corollary~3.8]{mms22} that 
\begin{equation}\label{sub_cplq}
 \sub h(u)=B^*\sub\big(  \d^2 g(\oz , \oy ) \big)(Bu).
\end{equation}
Note that no constraint qualification is needed for this chain rule to be valid. 
Appealing then to \cite[Proposition~13.9 \& Theorem~13.40]{rw}, we can calculate the graphical derivative of $\sub g$ at $\oz$ for $\oy$ as 
\begin{equation}\label{pwfor1}
 D(\sub g)(\oz, \oy)(w) = \sub\big(\tfrac{1}{2}\d^2 g(\oz , \oy )\big) (w) 
 =\bigcap_{i\in I(\oz, \oy, w)}\big\{v\in\R^m~\big|~v-A_i w\in N_{K_{C_i}(\oz,\oy_i)}(w)\big\}
\end{equation}
for any $w\in K_g(\oz,\oy)$, where $I(\oz, \oy, w):=\big\{i\in I(\oz)~\big|~w\in K_{C_i}(\oz,\oy_i)\big\}$.
 In particular, we have 
 \begin{equation}\label{pwfor2}
    D(\sub g)(\oz,\oy)(0)=K_g(\oz,\oy)^*.
 \end{equation}
We close this subsection with the reduction lemma for CPLQ functions, taken from \cite[Theorem~2.3]{s20}, that is essential in our subsequent analysis.
Let $(\oz,\oy)\in \gph \sub g$. The lemma states that graphs of the subgradient mapping $\sub g$, shifted to the origin, and of its graphical derivative at $\oz$ for $\oy$, $D(\sub g)(\oz, \oy)$, coincide locally around $(0, 0)$.
Namely, there exists a neighborhood ${\cal O}$ of $(0,0)\in \R^m\times \R^m$
for which we have
\begin{equation}\label{rl2}
\big ( (\gph \sub g) - (\oz,\oy)\big )\cap {\cal O}=\big (\gph D(\sub g)(\oz,\oy)\big )\cap {\cal O}.
\end{equation}

\subsection{An Iterative Framework for sSQP}\label{sect:fischer}

As shown in Section~\ref{sect:Intro}, the primal-dual sequence $\{(\xk,y^k)\b$, generated by the sSQP method, are solutions to a particular perturbation of the KKT system associated with the composite optimization problem \eqref{comp}, which is formulated as 
\begin{equation}\label{vs} 
  0=\nabla_x L(x,y),\;\;y\in \sub g(\Phi(x)).
\end{equation}
The latter can be written as a generalized equation 
\begin{equation}\label{geco} 
\begin{bmatrix}
0\\ 0
\end{bmatrix}\in  \begin{bmatrix}
\nabla_xL(x, y)\\ -\Phi(x)
\end{bmatrix} +
\begin{bmatrix}
0\\ (\partial g)^{-1}(y)
\end{bmatrix}.
\end{equation}
To elaborate more on the kind of perturbation we should consider for this generalized equation,  we briefly recall an abstract iterative framework, suggested by Fischer in \cite{Fis02}, for solving generalized equations with nonisolated solutions.
Given the mappings $\Psi: \H\to \H'$ and $G:\H \tto\H'$, where $\H$ and $\H'$ are finite dimensional Hilbert spaces,  consider the generalized equation formulated by
\begin{equation}\label{ge}
0\in \Psi(\bu) + G(\bu).
\end{equation}
Define the  solution mapping $\s:\H'\tto \H$ to the canonical perturbation of \eqref{ge} by    
\begin{equation}\label{solm}
\s(\bv)=\big\{\bu\in \H|\, \bv \in \Psi(\bu) +G(\bu)\big\}, \;\;\bv\in \H'.
\end{equation}
Iterative algorithms for solving \eqref{ge} often generate a sequence $\{\bu^k\b$ by iteratively solving subproblems of the form
\begin{equation}\label{age}
0\in \A(\bu, \bu^k)+G(\bu),
\end{equation}
in which  $\A:\H\times \H\to \H'$ is a certain linearization/approximation of the single-valued mapping $\Psi$ in \eqref{ge} around the current iterate, while the set-valued mapping is kept unchanged. 
In order to ensure the convergence and to establish the rate of convergence of $\{\bu^k\b$, we have to  choose a solution to \eqref{age} satisfying certain properties. 
Namely, given the current iterate $\bu^k$, we choose the next iterate $\bu^{k+1}$ sufficiently close to $\bu^k$ satisfying
\begin{equation}\label{ukk}
\bu^{k+1} \in \big\{ \bu\in \H\, \big|\, 0\in \A(\bu, \bu^k)+ G(\bu)\; \textrm{ and }\; \|\bu-\bu^k\|\leq c\, \dist(\bu^k,    \s(0))\big\},
\end{equation}
where $ c>0$ is arbitrary yet fixed, and $\s(0)$, taken from \eqref{solm} with $\bv=0$, is the set of solutions to \eqref{ge}. 
Below, we recall a result, established in \cite[Theorem~1]{Fis02} (see also \cite[Theorem~7.13]{is14}), in which the local convergence of the sequence $\{\bu^k\b$ was obtained under rather mild assumptions. 
 Recall that a set-valued mapping $F: \X\tto \Y$ is calm at $x$ for $y\in F(x)$ if there are neighborhoods $U$ of $x$ and $V$ of $y$ and a positive constant $\kappa$ for which we have 
$$
F(x')\cap V\subset F(x)+\kappa\,\|x'-x\|\B\quad \mbox{for all}\;\; x'\in U.
$$

\begin{Theorem}\label{thm:fischer} Let $\bar \bu \in \s(0)$, where $\s$ is the solution mapping from \eqref{solm}. Assume that $\s(0)$ is locally closed around $\bar \bu$ and that the following properties hold for some positive constant $c$:
\begin{enumerate}[noitemsep,topsep=2pt]
\item {\rm(}calmness of solution mapping{\rm)} the solution mapping $\s$ is calm at $0\in \H'$ for $\bar \bu$;

\item {\rm(}solvability of subproblems{\rm)} there exists a positive constant $\varepsilon_2$ such that for any $\tilde \bu \in \B_{\varepsilon_2}(\bar \bu)$   the localized solution set
\begin{equation*}
\big\{ \bu\in \H\, \big|\, 0\in \A(\bu, \tilde \bu)+ G(\bu)\; \textrm{ and }\; \|\bu-\tilde \bu\|\leq c\, \dist(\tilde \bu, \s(0))\big\}
\end{equation*}
is nonempty;

\item {\rm(}precision of approximation{\rm)} there exist a positive constant $\varepsilon_3$ such that 
\begin{equation*}
 \sup\big\{\|\Psi(\bu) - \A(\bu, \tilde \bu)\|~\big|~\bu\in \H\; \textrm{ with }\;\| \bu - \tilde \bu\|\leq  c\, \dist(\tilde \bu, \s(0))\big\}\leq o\big(\dist(\tilde \bu, \s(0))\big),
\end{equation*}
for all $\tilde \bu\in \B_{\varepsilon_3}(\bar \bu)$.
\end{enumerate}
Then there exists $\varepsilon_0>0$ such that for any starting point $\bu^0$ chosen from $\B_{\varepsilon_0}(\bar \bu)$, 
the iterative scheme in \eqref{ukk} 
generates a sequence $\{\bu^k\}_{k\in \N}$ that is convergent to some $\hat \bu \in \s(0)$. Moreover, the rate of convergence is Q-superlinear. 
\end{Theorem}

We next show that  the sSQP primal-dual sequence $\{(\xk,y^k)\b$ can fit into the iterative pattern described above.
To express the KKT system \eqref{vs} in the form of the generalized equation \eqref{ge}, we define the mappings $\Psi: \X\times \Y \to \X\times\Y$ and  $G: \X\times \Y \tto \X\times\Y$  by
\begin{equation}\label{ALMge}
\Psi(x, y) = \begin{bmatrix}
\nabla_xL(x, y)\\
-\Phi(x)
\end{bmatrix}
\quad\textrm{ and }\quad
G(x, y) =\begin{bmatrix}
0\\ (\partial g)^{-1}(y)
\end{bmatrix}.
\end{equation}
The KKT system  \eqref{sSQP-KKT} can be equivalently expressed in the form \eqref{age} as
\begin{equation}\label{age1}
   \begin{bmatrix}
       0\\0
   \end{bmatrix} \in \begin{bmatrix}
        \nabla_xL(\xk, \yk) + \nabla^2_{xx}L(\xk, \yk)(x-\xk) + \nabla \Phi(\xk)^*(y - \yk)\\
        -\Phi(\xk) - \nabla\Phi(\xk)(x-\xk) + \sigk(y - \yk)
    \end{bmatrix} + G(x, y),
\end{equation}
with $G$ defined in \eqref{ALMge}.
This suggests that the mapping $\A: \X\times \Y\times \X\times \Y \to \X\times \Y$ in \eqref{age} can be defined by
\begin{equation}\label{A}
\A(x, y, \tx, \ty) := \begin{bmatrix}
\nabla_xL(\tx, \ty) + \nabla^2_{xx}L(\tx, \ty)(x- \tx) + \nabla\Phi(\tx)^*(y - \ty)\\
-\Phi(\tx)  - \nabla\Phi(\tx)(x - \tx) + \sigma(\tx, \ty)(y - \ty)
\end{bmatrix}
\end{equation}
for a given pair $(\tx, \ty)\in \X\times \Y$ and $(x, y) \in \X\times \Y$, where $\sigma: \X\times\Y\to\R_+$ is appropriately chosen to provide a formula for the dual stabilization parameter for the next iterate in terms of the current one.
The next primal-dual iterate $(\xkk,\ykk)$, as a solution to \eqref{sSQP-KKT}, can now be chosen as a solution to the generalized equation
\begin{equation}\label{ge-A}
 (0,0)\in \A(x, y, \xk, \yk)+ G(x, y). 
\end{equation}
Set 
\begin{align*}
   (u, w)&:= \Psi(x, y) - \A(x, y, \tx, \ty)\\ 
   & = \Psi(x, y) - \Psi(\tx, \ty) - \nabla\Psi(\tx, \ty)\begin{bmatrix}
       x - \tx\\ y - \ty
   \end{bmatrix} + \begin{bmatrix}
       0\\ \sigma(\tx, \ty)(y-\ty)
   \end{bmatrix}.
\end{align*}
Differentiability of $\Psi$ then tells us that
\begin{equation*}
    \|(u, w)\| \leq \Big(\frac{o(\|(x-\tx, y - \ty)\|)}{\|(x-\tx, y - \ty)\|} + \sigma(\tx, \ty)\Big)\|(x-\tx, y - \ty)\|.
\end{equation*}
Thus, the estimate required in Theorem~\ref{thm:fischer}(c) is satisfied provided that $\|(x-\tx, y - \ty)\| \leq c\,\dist((\tx, \ty), \s(0, 0))$ and
\begin{equation*}
    \sigma(\tx, \ty) \leq \ell\,\dist((\tx, \ty), \s(0, 0)).
\end{equation*}
for some $\ell>0$.

\section{Solvability of Subproblems under SOSC}\label{sect:sosc}

Our main goal in this section is to demonstrate that the SOSC alone is sufficient for the solvability of the sSQP subproblems. 
To this end, we recall first some notation and error bound estimates that play an essential role in our analysis.

Given a point $\ox\in\R^n$, we define the set of {Lagrange multipliers} associated with  the KKT system \eqref{vs} at $\ox$ by
\begin{equation}\label{laset}
\Lambda(\ox):=\big\{y\in\R^m\;\big|\; 0\in \nabla_x L(\ox, y),\; y\in \sub g(\Phi(\ox))\big\}.
\end{equation}
If $(\ox,\oy)$ is a solution to the KKT  system \eqref{vs}, it is clear that   $\oy\in\Lambda(\ox)$.
Note that under a rather mild assumption such as the metric subregularity constraint qualification (MSCQ; see, \cite[Definition~3.2]{mms22})   the Lagrange multiplier set \eqref{laset} is nonempty when $\ox$ is a local minimizer of \eqref{comp}. Such a task requires a chain rule for the subdifferential of the composite function in \eqref{comp}, which can be found in \cite[Theorem~3.6]{mms22}. 
In this work, we do not focus on conditions ensuring the nonemptiness of $\Lm(\ox)$; rather, we always assume   that it is nonempty. 
Given  a solution $(\ox,\oy)$  to \eqref{vs}, the  SOSC for the composite problem \eqref{comp} at $(\ox,\oy)$ can be formulated (cf. see \cite[equation~(3.21)]{s20}) as  
\begin{equation}\label{sosc}
\begin{cases}
    \big\langle\nabla_{xx}^2L(\bar x,\oy)u,u\big\rangle+\d^2g (\Phi(\bar x),\oy) (\nabla \Phi(\ox)u )>0\\
    \textrm{for all } u \in \R^n\setminus\{0\} \textrm{ with } \nabla\Phi(\ox)u\in K_g \big(\Phi(\ox),\oy \big).
\end{cases}
\end{equation} 
 A fundamental special case of $g$ in \eqref{comp} arises when it is a polyhedral function. 
 For such $g$, \cite[Proposition~13.9]{rw} yields $\d^2g (\Phi(\ox),\oy)=\dd_{K_g (\Phi(\ox),\oy )}$, which allows the second subderivative term in \eqref{sosc} to be omitted.
 For instance, taking $g=\dd_{\R^{m_1}_-\times\{0\}^{m_2}}$ recovers a standard NLP with $m_1$ inequality and $m_2$ equality constraints. 
 In this setting, the SOSC in \eqref{sosc} reduces to the classical SOSC for NLPs.

For the composite problem \eqref{comp}, we define the  residual function $r:\R^n\times\R^m\to\R_+$ associated with the KKT system \eqref{vs} by
\begin{equation}\label{res}
r(x, y):= \|\nabla_x L(x, y)\|+\|\Phi(x)-{\rm prox}_g(\Phi(x)+ y)\|.
\end{equation}
As discussed below, this function provides an upper bound for the distance from nearby pairs $(x,y)$
to the set $\{\ox\}\times \Lm(\ox)$ under the SOSC.

\begin{Proposition}\label{prop:subr3}
   Assume that  $(\ox, \oy)$ is a solution to the KKT system \eqref{vs}
   and the SOSC in \eqref{sosc} holds. Then, there are positive constants $\tau_1\leq \tau_2$  such that 
\begin{equation}\label{subr3}
\tau_1\big(\|x-\ox\|+\dist(y,\Lambda(\ox))\big)\leq r(x, y)\leq \tau_2 \big(\|x-\ox\|+\dist(y,\Lambda(\ox))\big)
\end{equation}
for all $(x, y)$ in some neighborhood of $(\ox, \oy)$.
\end{Proposition}
\begin{proof}
The lower estimate comes from combining \cite[Theorem~3.6]{s20} and \cite[Propositions~3.8 \&~3.9]{s20}.
The upper estimate can be taken from \cite[Proposition~5.4]{HaS21} and it holds without requiring the SOSC.  
\end{proof}

Throughout the paper, we will fix a solution $(\ox, \oy)$ to the KKT system \eqref{vs} and consider an arbitrary function $\sigma:\R^n\times\R^m\to\R_+$ measuring the violation of optimality conditions in \eqref{vs}. Specifically, we assume the existence of a neighborhood $\mathcal{V}$ of $(\ox, \oy)$ and positive constants $\kappa_1\leq\kappa_2$ for which we have
\begin{equation}\label{eb}
\kappa_1\big(\|x-\ox\|+\dist(y,\Lambda(\ox))\big)\leq \sigma(x, y)\leq \kappa_2 \big(\|x-\ox\|+\dist(y,\Lambda(\ox))\big)\quad \textrm{for all }\; (x, y)\in \mathcal{V}.
\end{equation}
The upper estimate immediately implies that $\sigma(\ox, \oy)=0$ and $\sigma(x, y) \to 0$ as $(x, y)\to (\ox, \hat y)$ for some $\hat y \in \Lambda(\ox)$. According to Proposition~\ref{prop:subr3}, a natural choice for $\sigma$ is the residual function $r$ defined in \eqref{res}.

As stated at the beginning of this section, our goal is to establish the solvability of the sSQP subproblems under the SOSC in \eqref{sosc}.
Namely, given the current iterate $(x^k,y^k)$, we will show that the variational system in \eqref{sSQP-KKT} has a solution satisfying a certain local condition.
Our strategy is structured as follows: we first demonstrate that the auxiliary variational system 
$$
\begin{cases}
0= \nabla\ph(\xk)  + \nabla^2_{xx}L(\xk, \yk)(x-\xk) + \nabla \Phi(\xk)^*q,\\
q - \oy \in \sub\big(\tfrac{1}{2}\d^2g(\oz, \oy)\big)\big(\Phi(\xk) + \nabla\Phi(\xk)(x-\xk) - \sigma_k(y - \yk)-\oz\big)
\end{cases}
$$
admits a solution. 
This result is then combined with the reduction lemma for CPLQ functions in \eqref{rl2} to guarantee the solvability of \eqref{sSQP-KKT}. 
While the idea of employing an auxiliary variational system is borrowed from \cite{w98, ha98, FeS10} in the context of NLPs, our approach differs fundamentally in using the  reduction lemma to transition from the auxiliary variational system to \eqref{sSQP-KKT}. 
We start with the following result, which demonstrates that a perturbed variant of the SOSC in \eqref{sosc} holds in a neighborhood of the reference KKT point $(\ox,\oy)$, thus generalizing a similar result for NLPs in \cite[Proposition~1]{FeS10}.

\begin{Proposition}\label{ex_sosc}
Let $(\ox, \oy)$ be a solution to the KKT system \eqref{vs} such that the SOSC \eqref{sosc} holds at $(\ox, \oy)$ and $\sigma$ satisfy the upper bound in \eqref{eb}. Set $\oz :=\Phi(\ox)$. Then there exist positive constants $\varepsilon$ and $\gamma$ such that for all $(x, y) \in \B_\varepsilon(\ox, \oy))$ we have
\begin{equation}\label{s7}
\langle \nabla^2_{xx}L(x, y)u, u\rangle+\d^2g(\oz, \oy)\big(\nabla\Phi(x)u-\sigma(x, y)w\big)+\sigma(x, y)\|w\|^2 \geq \gamma(\|u\|^2+\sigma(x, y)\|w\|^2)
\end{equation}
for all $(u, w) \in \K(x, y):= \big\{(u, w)~\big|~ \nabla\Phi(x)u-\sigma(x, y)w\in K_g(\oz, \oy)\big\}$.
\end{Proposition}

\begin{proof}
Suppose by contradiction that there exist sequences $(\xk, \yk) \to (\ox, \oy)$ and $(\uk, \wk)\in \K(\xk, \yk)$ satisfying
\begin{equation}\label{s1}
\langle \Hk\uk, \uk\rangle+\d^2g(\oz, \oy)\big(\nabla\Phi(\xk)\uk-\sigk\wk\big)+\sigk\|\wk\|^2 < \frac{1}{k}(\|\uk\|^2+\sigk\|\wk\|^2),
\end{equation}
with $\Hk:= \nabla^2_{xx}L(\xk, \yk)$ and $\sigk:= \sigma(\xk, \yk)$. Set $\beta_k:=\sqrt{\|\uk\|^2+\sigk\|\wk\|^2}$. The latter estimate obviously yields $\beta_k>0$. Passing to a subsequence, if necessary, we can assume that
\begin{equation*}
\Big(\frac{\uk}{\beta_k}, \frac{\sqrt{\sigk}\wk}{\beta_k}\Big)\to (u, w)\;\;\textrm{as }\, k\to \infty\;\;\textrm{for some }\, (u, w)\neq 0.
\end{equation*}
Recall that $\nabla\Phi(\xk)\uk-\sigk\wk \in K_g(\oz, \oy)$ for all $k\in \N$. Dividing the left-hand side of the latter by $\beta_k>0$ and passing through the limit as $k\to\infty$, we get  $\nabla\Phi(\ox)u\in K_g(\oz, \oy)$, due to $\sigk\to 0$ as $k\to\infty$. 
It is worth mentioning that $\d^2g(\oz, \oy)$ is homogeneous of order 2 and convex (cf. \cite[Propositions~13.5 \& 13.20(a)]{rw}. It is therefore continuous relative to its domain $\dom \d^2g(\oz, \oy) = K_g(\oz, \oy)$. Dividing both sides of \eqref{s1} by $\beta_k^2$ gives us
\begin{align*}
\Big\langle \Hk\frac{\uk}{\beta_k}, \frac{\uk}{\beta_k}\Big\rangle +\d^2g(\oz, \oy)\Big(\nabla\Phi(\xk)\frac{\uk}{\beta_k}-\frac{\sigk\wk}{\beta_k}\Big)+\frac{\sigk\|\wk\|^2}{\beta_k^2}
< \frac{1}{k},
\end{align*}
Passing through the limit as $k\to\infty$ then yields
\begin{equation*}
\langle\nabla_{xx}^2L(\bar x,\oy)u,u\rangle+\d^2g (\oz,\oy) (\nabla \Phi(\ox)u)+\|w\|^2\leq 0.
\end{equation*}
The SOSC \eqref{sosc} then implies that $u = 0$, and therefore $w = 0$, a contradiction. The proof is then complete.
\end{proof}

As a direct consequence of Proposition~\ref{ex_sosc}, we establish below the nonsingularity of certain matrices arising in the variational system \eqref{sSQP-KKT}. 

\begin{Corollary}\label{cor:SOSC}
Let $(\ox, \oy)$ be a solution to the KKT system \eqref{vs} such that the SOSC \eqref{sosc} holds at $(\ox, \oy)$ and $\sigma$ satisfy the upper bound in \eqref{eb}. Then there exists $\varepsilon>0$ such that for all $(x, y) \in \B_\varepsilon(\ox, \oy)$ satisfying $\sigma(x, y)>0$ the matrix
\begin{equation}\label{s2}
\begin{bmatrix}
\nabla^2_{xx} L(x, y)&\nabla\Phi(x)^*\\
-\nabla\Phi(x)&\sigma(x, y)I
\end{bmatrix}
\end{equation}
is nonsingular.
\end{Corollary}

\begin{proof}
Let $\varepsilon>0$ and $\gamma>0$ be such that \eqref{s7} holds for all $(x, y)\in \B_\varepsilon(\ox, \oy)$ and all $(u, w)\in \K(x, y)$. 
Fix $(x, y) \in \B_\varepsilon(\ox,\oy)$ with $\sigma(x, y)>0$ and let $(u, w)$ be in the kernel of the matrix defined in \eqref{s2}. That means
\begin{equation}\label{s4}
\begin{cases}
\nabla^2_{xx} L(x, y)u+\nabla\Phi(x)^*w = 0,\\
\nabla\Phi(x)u - \sigma(x, y)w = 0.
\end{cases}
\end{equation}
Multiplying both sides of the first equality by $u$ and taking the second one into account, we obtain
\begin{equation}\label{s3}
0 = \langle \nabla^2_{xx} L(x, y)u, u\rangle+\langle\nabla\Phi(x)u, w\rangle =  \langle \nabla^2_{xx} L(x, y)u, u\rangle+\sigma(x, y)\|w\|^2
\end{equation}
Moreover, the second equality of \eqref{s4} particularly implies $\d^2g(\Phi(\ox), \oy)(\nabla\Phi(x)u-\sigma(x, y)w)=0$ and $(u, w)\in \K(x, y)$. Thus, combining \eqref{s7} and \eqref{s3}, we get
\begin{equation*}
\|u\|^2+\sigma(x, y)\|w\|^2=0,
\end{equation*}
which yields $(u, w)=(0,0)$, since $\sigma(x, y)>0$, and therefore the matrix in \eqref{s2} is nonsingular. The proof is then complete.
\end{proof}

We continue with another observation that guarantees the existence of a solution to the  variational system \eqref{sSQP-KKT}. 
Specifically, the perturbed problem \eqref{sSQP-1} is deliberately designed to ensure that the auxiliary variational system \eqref{au-KKT} has a solution.

\def\hu{\hat{u}}
\def\hw{\hat{w}}
\begin{Proposition}\label{prop:existence}
Let $(\ox, \oy)$ be a solution to the KKT system \eqref{vs} such that the SOSC \eqref{sosc} holds at $(\ox, \oy)$ and $\sigma$ satisfy the upper bound in \eqref{eb}. Let $\oz := \Phi(\ox)$.
Then there exist $\varepsilon>0$ such that for all $(x, y) \in \B_\varepsilon(\ox,\oy)$ with $\sigma(x, y)>0$, the following auxiliary optimization problem
\begin{align}\label{sSQP-1}
    \mini_{(s, q) \in \X\times \Y} &\;\la\nabla \ph(x),s-x\ra + \tfrac{1}{2}\la\nabla^2_{xx}L(x, y)(s - x), s - x\ra \nonumber\\
    &\; +\la \oy, \Phi(x) + \nabla\Phi(x)(s - x) - \sigma(x, y)(q - y) -\oz\ra\\
    &\; + \tfrac{1}{2}\d^2g(\oz, \oy)\big(\Phi(x) + \nabla\Phi(x)(s - x) - \sigma(x, y)(q - y) - \oz\big) + \tfrac{1}{2}\sigma(x, y)\|q\|^2\nonumber
\end{align}
has a nonempty compact solution set.
\end{Proposition}

\begin{proof}
Let $\varepsilon>0$ and $\gamma>0$ be such that \eqref{s7} holds for all $(x, y)\in \B_\varepsilon(\ox, \oy)$ and all $(u, w)\in \K(x, y)$. 
Fix $(x, y)\in\B_\varepsilon(\ox, \oy)$ with $\sigma(x, y) > 0$ and let $(\hu, \hw)$ be the unique solution to the linear system
\begin{equation*}
\begin{bmatrix}
\nabla^2_{xx} L(x, y)&\nabla\Phi(x)^*\\
-\nabla\Phi(x)&\sigma(x, y) I
\end{bmatrix}
\begin{bmatrix} u\\ w\end{bmatrix}
=\begin{bmatrix}
-\nabla_x L(x, y)\\ \Phi(x)-\oz
\end{bmatrix},
\end{equation*}
which exists due to nonsingularity of the matrix on the left-hand side, established in Corollary \ref{cor:SOSC}. Define $(u, w):= (s-x-\hu, q-y-\hw)$ and observe that
\begin{equation*}
\Phi(x)+\nabla\Phi(x)(s-x)-\sigma(x, y)(q-y) - \oz
= \nabla\Phi(x)u-\sigma(x, y) w
\end{equation*}
and
\begin{align*}
   &\langle \nabla\varphi(x), s-x\rangle +\tfrac{1}{2}\langle \nabla^2_{xx} L(x, y)(s-x), s-x\rangle + \tfrac{1}{2}\sigma(x, y)\|q\|^2\\
      &=  \tfrac{1}{2}\la \nabla^2_{xx} L(x, y)u, u\ra + \tfrac{1}{2}\sigma(x, y)\|w\|^2 -  \la y + \hw, \nabla\Phi(x)u - \sigma(x, y) w\ra + c,
\end{align*}
with $c = c (x, y, \hu, \hw)$ being independent of $(s, q)$.
Problem \eqref{sSQP-1} therefore can be reformulated, using variable $(u, w)$, as follows
 \begin{align}\label{s6}
\mini_{(u,w) \in \X\times \Y}&\; \tfrac{1}{2}\Big[\langle \nabla^2_{xx} L(x, y)u, u\rangle+\d^2g(\oz, \oy)(\nabla\Phi(x)u-\sigma(x, y) w)+\sigma(x, y)\|w\|^2\Big]\nonumber\\
& - \langle y + \hw - \oy, \nabla\Phi(x)u -\sigma(x, y) w\rangle.
\end{align}
Let $\ell(u, w)$ denote the objective function in \eqref{s6}.
Employing the estimate \eqref{s7}, we get
\begin{equation}\label{s14}
\ell(u, w)\geq \frac{\gamma}{2}(\|u\|^2+\sigma(x, y)\|w\|^2) - \langle y + \hw - \oy, \nabla\Phi(x)u -\sigma(x, y) w\rangle
\end{equation}
for all $(u, w)\in \K(x, y)$. This inequality is evidently true for $(u, w)\notin \K(x, y)$. It follows from  the Cauchy-Schwarz inequality that
\begin{align*}
- \langle y + \hw - \oy, \nabla\Phi(x)u &-\sigma(x, y) w\rangle  = - \langle \nabla\Phi(x)^*(y + \hw - \oy), u\ra +\sigma(x, y)\la y + \hw - \oy,  w\rangle\\
 \geq &- \|y + \hw - \oy\|\sqrt{\| \nabla\Phi(x)\|^2+\sigma(x, y)}\cdot\sqrt{\|u\|^2+\sigma(x, y)\|w\|^2},
\end{align*}
or, equivalently,
\begin{equation*}
-\frac{\langle y + \hw - \oy, \nabla\Phi(x)u -\sigma(x, y) w\rangle}{\|u\|^2+\sigma(x, y)\|w\|^2}\geq -\frac{ \|y + \hw - \oy\|\sqrt{\| \nabla\Phi(x)\|^2+\sigma(x, y)}}{\sqrt{\|u\|^2+\sigma(x, y)\|w\|^2}}.
\end{equation*}
Observe that the right-hand side goes to $0$ as $\|(u, w)\|\to \infty$. Thus, for $(u, w)$ with $\|(u, w)\|$ sufficiently large, we have
\begin{equation*}
 - \langle y + \hw - \oy, \nabla\Phi(x)u -\sigma(x, y) w\rangle \geq -\frac{\gamma}{4}(\|u\|^2+\sigma(x, y)\|w\|^2).
\end{equation*}
This, together with \eqref{s14}, proves coercivity of the objective function $\ell(u, w)$. 
We can then find a large positive constant $M$ such that solving \eqref{s6} is equivalent to minimizing $\ell(u, w)$ over the ball $M\B\subset\X\times\Y$. The just described optimization problem obviously has nonempty and compact solution set, due to the classical Weierstrass theorem, and so does the problem \eqref{sSQP-1}. The proof is then complete.
\end{proof}

We now derive the solvability of an auxiliary variational system closely related to \eqref{sSQP-KKT}. Note that its proof heavily relies on Proposition~\ref{prop:existence} and the chain rule \eqref{sub_cplq} for composite CPLQ functions. 

\begin{Proposition}\label{prop:solv}
Let $(\ox, \oy)$ be a solution to the KKT system \eqref{vs} such that the SOSC \eqref{sosc} holds at $(\ox, \oy)$ and let $\sigma$ satisfy the upper bound in \eqref{eb}. Then, there exists a positive constant $\varepsilon$   such that for all $(x, y)\in \B_\varepsilon(\ox, \oy)$ with $\sigma(x, y)>0$,  the solution set to the variational system
\begin{equation}\label{au-KKT}
\begin{cases}
\quad \;\;0 = \nabla\varphi(x)+\nabla^2_{xx}L(x, y)(s-x) + \nabla\Phi(x)^*q,\\
q - \oy \in \sub\big(\tfrac{1}{2}\d^2g(\oz, \oy)\big)\big(\Phi(x)+\nabla\Phi(x)(s-x)-\sigma(x, y)(q-y)-\oz\big).
\end{cases}
\end{equation}
is always nonempty.
\end{Proposition}
\begin{proof} Take the positive constant $\ve$ from  Proposition~\ref{prop:existence} and let $(x, y)\in \B_\varepsilon(\ox, \oy)$ with $\sigma(x, y)>0$. 
According to Proposition~\ref{prop:existence}, the auxiliary optimization problem \eqref{sSQP-1}  has a solution, says $(s, q)$. Thus, the first-order optimality condition   for \eqref{sSQP-1} can be expressed via the chain rule in \eqref{sub_cplq} as 
\begin{equation*}
\begin{cases}
   0 = \nabla\varphi(x)+\nabla^2_{xx}L(x, y)(s-x) + \nabla\Phi(x)^*(\oy +\nu),\\
   0 = \sigma(x, y)(q - \oy) - \sigma(x, y) \nu,\\ 
    \nu \in \sub\big(\tfrac{1}{2}\d^2g(\oz, \oy)\big)\big(\Phi(x)+\nabla\Phi(x)(s-x)-\sigma(x, y)(q-y)-\oz\big),
\end{cases}
\end{equation*}
The second equality and $\sigma(x, y)>0$  yield $\nu = q - \oy$, which in turn tells us that $(s, q)$ indeed solves the  variational system \eqref{au-KKT} and thus completes the proof.
\end{proof}

A similar result was established in \cite[Proposition~2]{FeS10} for NLPs under the classical SOSC. 
The approach therein  relies heavily on the existence of solutions for a class of generalized equations from \cite[Theorem~2.5.10]{fp}. 
Our proof of Proposition~\ref{prop:solv}, however, departs significantly from this approach and employs the fact that the optimization problem \eqref{sSQP-1} always admits a solution.  

To establish the solvability of the KKT systems associated with the sSQP subproblems under the hypothesis of Proposition~\ref{prop:solv}, it suffices to show, for any $(x, y)\in \B_\varepsilon(\ox, \oy)$ with $\sigma(x, y)>0$, that the variational system 
\begin{equation}\label{ssqp-kkt}
\begin{cases}
        0= \nabla\ph(x) + \nabla^2_{xx}L(x, y)(s-x)+ \nabla \Phi(x)^*q ,\\
         q\in \sub g\big(\Phi(x) + \nabla\Phi(x)(s-x) - \sigma(x, y)(q-y)\big),
    \end{cases}
\end{equation}
admits a solution $(s,q)$. 
This task can be accomplished by showing that any solution $(s,q)$ to \eqref{au-KKT} remains  sufficiently close to the reference KKT point $(\ox,\oy)$.
Indeed, if this holds, combining the reduction lemma for CPLQ functions in \eqref{rl2} with the first equality in \eqref{pwfor1} yields the solvability of \eqref{ssqp-kkt}. 
The next result establishes an error bound estimate for the solution to \eqref{au-KKT}, guaranteeing that any solution stays within a target neighborhood of $(\ox,\oy)$.

\begin{Proposition}\label{prop:eb}
Let $(\ox, \oy)$ be a solution to the KKT system \eqref{vs} such that the SOSC \eqref{sosc} holds at $(\ox, \oy)$ and $\sigma$ satisfy both error bound estimates in \eqref{eb}. Then there exist positive constants $\varepsilon$ and $\kappa$ such that for all $(x, y)\in \B_\varepsilon(\ox, \oy)$ with $\sigma(x, y)>0$, the error bound estimate
\begin{equation}\label{eb3}
\|s-x\|+\|q-y\|\leq \kappa\sigma(x, y)
\end{equation}
holds for any solution $(s, q)$ to the system \eqref{au-KKT}.
\end{Proposition}

\begin{proof}
Assume by contradiction that the estimate \eqref{eb3} fails. Then there exist sequence $(\xk, \yk)\to (\ox, \oy)$, with $\sigma_k:= \sigma(\xk, \yk)>0$, and $(s^k, q^k)$ such that
\begin{equation}\label{s8}
\|s^k-\xk\|+\|q^k-\yk\|>k \sigma_k,
\end{equation}
where $(s^k, q^k)$ satisfies
\begin{equation}\label{s9}
    0=\nabla\varphi(\xk)+ \Hk(s^k-\xk)+\nabla\Phi(\xk)^*q^k,\quad
q^k - \oy \in \sub\big(\tfrac{1}{2}\d^2g(\oz, \oy)\big)(\zk - \oz),
\end{equation}
where $\Hk := \nabla^2_{xx}L(\xk, \yk)$, $\zk :
= \Phi(\xk)+\nabla\Phi(\xk)(s^k-\xk)-\sigma_k(q^k-\yk)$, and $\oz := \Phi(\ox)$.
Denoting by $\beta_k := \|s^k-\xk\|+\|q^k-\yk\|$, we infer via \eqref{s8} that $\sigma_k = o(\beta_k)$ as $k\to\infty$. By the lower estimate in \eqref{eb} and the fact that $(\xk, \yk)\to (\ox, \oy)$ as $k\to\infty$, we then obtain  
\begin{equation}\label{s10}
\|\xk -\ox\| = o(\beta_k)\quad\textrm{and}\quad \|\yk -\hyk\|=o(\beta_k)\quad\textrm{as } k\to\infty,
\end{equation}
where $\hyk := P_{\Lm(\ox)}(\yk)$. Passing to a subsequence, if necessary, we can assume that
\begin{equation}\label{s11}
\Big(\frac{s^k-\xk}{\beta_k}, \frac{q^k-\yk}{\beta_k}\Big)\to (u, w)\;\;\textrm{as }\, k\to \infty\;\;\textrm{for some }\, (u, w)\neq 0.
\end{equation}
Recalling that $\hyk\in \Lm(\ox)$, we can then equivalently rewrite the first equality in \eqref{s9} as
\begin{equation*}
0=\nabla_x L(\xk, \yk) - \nabla_x L(\ox, \hyk) +\Hk(s^k-\xk)+\nabla\Phi(\xk)^*(q^k-\yk).
\end{equation*}
Dividing both sides by $\beta_k>0$,   passing through the limit as $k\to\infty$, and taking  into account \eqref{s10} and \eqref{s11}, we arrive at
\begin{equation}\label{s12}
0=\nabla^2_{xx}L(\ox, \oy)u+\nabla\Phi(\ox)^*w,
\end{equation}
which yields
\begin{equation}\label{s13}
0=\la\nabla^2_{xx}L(\ox, \oy)u, u\ra+\la\nabla\Phi(\ox)u, w\ra.
\end{equation}
On the one hand, we observe from the second relation in \eqref{s9} that $\zk - \oz \in \dom \d^2g(\oz, \oy) = K_g(\oz, \oy)$.
Thus, we get via \eqref{s10}, \eqref{s11}, and $\sigk = o(\beta_k)$ that
\begin{equation}\label{s16}
\frac{z^k-\oz}{\beta_k} = \frac{\Phi(\xk)-\Phi(\ox)}{\beta_k}+\nabla\Phi(\xk)\frac{s^k-\xk}{\beta_k}-\frac{\sigma_k}{\beta_k}(q^k-\yk)\to \nabla\Phi(\ox)u\; \; \textrm{as }\; k\to\infty,
\end{equation}
and arrive at $\nabla\Phi(\ox)u\in K_g(\oz, \oy)$. 
On the other hand, since that $\d^2g(\oz, \oy)$ is proper, convex, and positively homogeneous of degree $2$, the second inclusion in \eqref{s9} leads us to
\begin{equation*}
\d^2g(\oz, \oy)\Big(\frac{z^k-\oz}{\beta_k}\Big)=\Big\la \frac{z^k-\oz}{\beta_k},\frac{q^k - \oy}{\beta_k}\Big\ra,
\end{equation*}
see, e.g., the proof of \cite[Proposition~3.9]{s20}.
This equality in turn implies that
\begin{equation}\label{s15}
\d^2g(\oz, \oy)\Big(\frac{z^k-\oz}{\beta_k}\Big)=\Big\la \frac{z^k-\oz}{\beta_k},\frac{q^k-\hyk}{\beta_k}\Big\ra+ \frac{1}{\beta_k^2}\la z^k-\oz,\, \hyk - \oy\ra
\leq \Big\la \frac{z^k-\oz}{\beta_k},\frac{q^k-\hyk}{\beta_k}\Big\ra,
\end{equation}
due to $z^k - \oz \in K_g(\oz, \oy) = N_{\sub g(\oz)}(\oy)$ and $\hyk\in \sub g(\oz)$.
Passing to the limit as $k\to\infty$ in \eqref{s15}  and using to \eqref{s11} and \eqref{s16} bring us to 
\begin{equation*}
\d^2g(\oz, \oy)\big(\nabla\Phi(\ox)u\big)\leq \la \nabla\Phi(\ox)u, w\ra.
\end{equation*}
Combining this with \eqref{s13}, we arrive at
\begin{equation*}
\la\nabla^2_{xx}L(\ox, \oy)u, u\ra+\d^2g(\oz, \oy)\big(\nabla\Phi(\ox)u\big)\leq 0,
\end{equation*}
and which implies via SOSC \eqref{sosc} and $\nabla\Phi(\ox)u\in K_g(\oz, \oy)$ that $u=0$. 
Consequently, we get from \eqref{s12} that $\nabla\Phi(\ox)^*w=0$. 

Denote $E:= \ker\nabla\Phi(\ox)^*$. 
We finish the proof by showing that $w=P_E(w)=0$. 
Recall that $z^k - \oz\in K_g(\oz, \oy)$ with $K_g(\oz, \oy)$ being described in \eqref{domcr}.
Passing to a subsequence,  if necessary, we can assume without loss of generality that there is some index $i_0\in I(\oz)$ such that $z^k-\oz \in K_\Ci(\oz, \oy_\ii)$ for all $k$, where $\oy_\ii:= \oy-A_\ii\oz-a_\ii$. Since $K_\Ci(\oz, \oy_\ii)$ is a polyhedral cone, it has the representation
\begin{equation}\label{s21}
K_\Ci(\oz, \oy_\ii) = \big\{ v\in \R^m\, \big|\, \la c^j, v\ra\leq 0, \; j\in J_0\big\}
\end{equation}
for some  $ c^j \in \R^m, j \in J_0$, a finite index set. Denote by $J_k$ the set of the indexes of active inequalities that define $K_\Ci(\oz, \oy_\ii)$, in \eqref{s21}, at $z^k-\oz$ and assume without loss of generality that $J_k = J\subset J_0$ for all $k$. Define
\begin{equation*}
\mathcal{B}:=\big\{v\in K_\Ci(\oz, \oy_\ii) \, \big|\, \la c^j, v\ra =0, \; j\in J\big\}
\end{equation*}
and observe that
\begin{equation}\label{s20a}
  \mathcal{B}^*=\spann\big\{c^j \,\big|\, j\in J\big\}+\cone\big\{c^j \,\big|\, j\in J_0\setminus J\big\}.
\end{equation}
Clearly, $\mathcal{B}$ is a cone with $z^k-\oz\in \mathcal{B}$ for all $k$, and therefore
\begin{equation}\label{s17}
P_E\Big(\frac{z^k-\oz}{\sigma_k\beta_k}\Big)\in P_E(\mathcal{B}).
\end{equation}
It is worth noting that $P_E$ is a linear mapping and $P_E(\nabla\Phi(\ox)\R^n)=\{0\}$. Since $\Phi$ is $\mathcal{C}^2$-smooth, we have
\begin{equation*}
\frac{\Phi(\xk)-\Phi(\ox)}{\sigma_k\beta_k}=\frac{1}{\sigma_k\beta_k}\nabla\Phi(\ox)(\xk - \ox)+O\Big(\frac{\|\xk - \ox\|^2}{\sigma_k\beta_k}\Big).
\end{equation*}
which in turn, by the fact that $\|\xk -\ox\|= O(\sigk)=o(\beta_k)$, yields
\begin{equation}\label{s18}
P_E\Big(\frac{\Phi(\xk)-\Phi(\ox)}{\sigma_k\beta_k}\Big)= P_E\Big(O\Big(\frac{\|\xk - \ox\|^2}{\sigma_k\beta_k}\Big)\Big) \to 0 \quad\textrm{as } k\to\infty.
\end{equation}
Moreover, by the same fact and $s^k - \xk = o(\beta_k)$, we also obtain  
\begin{equation}\label{s19}
P_E\Big(\frac{1}{\sigma_k\beta_k}\nabla\Phi(\xk)(s^k-\xk)\Big) = P_E\Big(\frac{\nabla\Phi(\xk)-\nabla\Phi(\ox)}{\sigma_k}\frac{s^k-\xk}{\beta_k}\Big) \to 0\quad\textrm{as } k\to\infty.
\end{equation}
Combining \eqref{s11}, \eqref{s16}, and \eqref{s17}--\eqref{s19}, we arrive at $-P_E(w)\in P_E(\mathcal{B})$. Since $w\in E$, the latter inclusion confirms that 
\begin{equation*}
-w =-P_E(w) =P_E(b)\quad\textrm{for some}\quad b\in \mathcal{B}. 
\end{equation*}
Moreover, observe from the second relation in \eqref{s9} and \eqref{pwfor1} that
\begin{equation}\label{s22a}
q^k-\oy \in D(\sub g)(\oz, \oy)(z^k-\oz)\subset \big\{v\in\R^m\, \big|\, v-A_\ii(z^k-\oz)\in N_{K_\Ci(\oz,\, \oy_\ii)}(z^k-\oz)\big\},
\end{equation}
where we utilize the  fact that $z^k-\oz\in K_\Ci(\oz, \oy_\ii)$ for all $k$. Taking into account \eqref{s21} and the definition of the index set $J$, we get from \eqref{s22a} that
\begin{equation}\label{s20}
q^k - \oy - A_\ii(z^k-\oz)\in N_{K_\Ci(\oz,\, \oy_\ii)}(z^k-\oz) = \cone\big\{c^j \,\big|\, j\in J\big\}.
\end{equation}
On the other hand, since $\hyk, \oy\in\sub g(\oz)$ and $\hyk \to \oy$ as $k\to\infty$, we conclude from the reduction relation in \eqref{rl2} and the formulas in \eqref{domcr} and \eqref{pwfor2} that
\begin{equation*}
\hyk-\oy \in \big(K_g(\oz, \oy)\big)^*\subset \big(K_\Ci(\oz, \oy_\ii)\big)^* = \cone\big\{c^j \,\big|\, j\in J_0\big\},
\end{equation*}
where we use \eqref{s21} to derive the last equality.
Combining this and \eqref{s20} and using \eqref{s20a} result in
\begin{align*}
\hyk-q^k+A_\ii(z^k-\oz) &\in \cone\big\{c^j \,\big|\, j \in J_0\big\}-\cone\big\{c^j \,\big|\, j\in J\big\}\\
&=\spann\big\{c^j \,\big|\, j\in J\big\}+\cone\big\{c^j \,\big|\, j\in J_0\setminus J\big\} = \mathcal{B}^*.
\end{align*}
Recall that  $\beta_k>0$. Thus,  we have
\begin{equation*}
-\frac{q^k-\hyk}{\beta_k}+A_\ii\frac{z^k-\oz}{\beta_k}\in \mathcal{B}^*.
\end{equation*}
Passing to the limit as $k\to\infty$ and taking into account \eqref{s10}, \eqref{s11}, and \eqref{s16} with $u=0$, we deduce that $-w \in \mathcal{B}^*$.
In summary, we have shown that
\begin{equation*}
\begin{cases}
-w =-P_E(w) =P_E(b), \; \; b\in \mathcal{B},\\
-w\in \mathcal{B}^*.
\end{cases}
\end{equation*}
This clearly tell us that  
\begin{equation*}
0\leq \la w, b\ra = \la P_E(w), b\ra =\la w, P_E(b)\ra=\la w, -w\ra = -\|w\|^2,
\end{equation*}
and therefore $w=0$. Since $u=0$, the latter clearly contradicts \eqref{s11} and thus completes the proof.
\end{proof}

We are now in a position to present our main result in which the existence of solutions
to the variational system in \eqref{sSQP-KKT} is justified. Furthermore, it is shown that the primal-dual solution to \eqref{sSQP-KKT} can be chosen such that an error bound estimate for the consecutive sSQP iterates holds. The latter is central to ensure  the solvability  in Theorem~\ref{thm:fischer}(b) with a local condition for the generalized equation in \eqref{ge-A} associated with the proposed sSQP method.

\begin{Theorem}\label{thm:existence}
Let $(\ox, \oy)$ be a solution to the KKT system \eqref{vs} such that the SOSC \eqref{sosc} holds at $(\ox, \oy)$ and $\sigma$ satisfy both error bound estimates in \eqref{eb}. Then there exist positive constants $\varepsilon$ and $\kappa$ such that for all $(x, y)\in \B_\varepsilon(\ox, \oy)$ with $\sigma(x, y)>0$, there exists a solution $(s, q)$ to the variational system in \eqref{ssqp-kkt}, 
satisfying the error bound estimate
\begin{equation}\label{eb3.1}
\| s-x\|+\| q-y\|\leq \kappa\sigma(x, y).
\end{equation}
\end{Theorem}

\begin{proof}
Take the constants $\varepsilon, \kappa>0$  so that the conclusions in  Propositions~\ref{prop:solv} and ~\ref{prop:eb} hold. Thus, for any  $(x, y) \in \B_\varepsilon(\ox, \oy)$ with $\sigma(x, y)>0$, we can find a solution $(s, q)$ to \eqref{au-KKT}  that satisfies the error bound estimate in \eqref{eb3}.
Set $z := \Phi(x) + \nabla\Phi(x)(s-x) - \sigma(x, y)(q - y)$ and $\oz := \Phi(\ox)$.
Since $\sigma(x, y)\to 0$ as $(x, y) \to (\ox, \oy)$ and $\Phi$ is Lipschitz continuous, by shrinking the ball $\B_\varepsilon(\ox, \oy)$, if necessary, we can assume that $(z - \oz, q - \oy) \in {\cal O}$ for any  $(x, y) \in \B_\varepsilon(\ox, \oy)$, where ${\cal O}$ is the neighborhood of the origin in $\Y\times \Y$ taken from \eqref{rl2}.
Utilizing \eqref{pwfor1}, we then get from the second inclusion in \eqref{au-KKT} that
\begin{equation*}
    (z - \oz, q - \oy) \in \big(\gph D(\sub g)(\oz, \oy)\big) \cap {\cal O}.
\end{equation*}
The reduction formula \eqref{rl2} then tells us $(z, q) \in \gph \sub g$, or, equivalently, 
\begin{equation*}
    q \in \sub g\big(\Phi(x) + \nabla\Phi(x)(s-x) - \sigma(x, y)(q - y)\big).
\end{equation*}
This, together with the first equality in \eqref{au-KKT}, proves that $(s, q)$ indeed solves \eqref{ssqp-kkt}.
Since $(s, q)$ satisfies the estimate in \eqref{eb3}, we obtain \eqref{eb3.1}, which completes the proof.
\end{proof}

Theorem~\ref{thm:existence} extends the solvability and the error bound analysis of \cite[Theorem~3]{FeS10} for the classical NLPs to the composite optimization in \eqref{comp}. 
Beyond this generalization, it constitutes a pivotal step toward understanding the local convergence behavior of the sSQP method for the composite problem in \eqref{comp} and opens the door to investigate  similar results for conic and composite optimization problems where the modeling function does not have a polyhedral structure.

\section{Solvability of Subproblems under Noncriticality of Lagrange Multipliers}\label{sec:noncri}

While the SOSC, used in Section~\ref{sect:sosc}, is relatively mild assumption, it is tempting to look for possible ways to weaken this assumption. Our motivation comes from an observation in \cite{is12}, where it was shown that the local superlinear convergence of the sSQP method for NLPs with only equality constraints   can be ensured under a weaker assumption, called  noncriticality of  Lagrange multipliers. 
Motivated by this result, we investigate, in the present section, whether an analogous replacement can be extended to the composite problem \eqref{comp}. To this end, we first recall the concept of noncriticality of Lagrange multipliers for \eqref{comp} as introduced in \cite{s20}, which generalizes the original formulation for NLPs in \cite{iz08}.

\begin{Definition}
Let $(\ox,\oy)$ be a solution to the KKT system \eqref{vs}. Then the multiplier $\oy$ is said to be {critical} for \eqref{vs} if there is a nonzero vector $u\in\R^n$  satisfying the inclusion
\begin{equation}\label{crc}
0\in\nabla^2_{xx}L (\ox,\oy)u+\nabla \Phi(\ox)^*D(\sub g) (\Phi(\ox),\oy ) (\nabla \Phi(\ox)u ).
\end{equation}
The multiplier $\oy$ is {noncritical} for \eqref{vs} if \eqref{crc} admits only the trivial solution $u=0$. 
\end{Definition}

As established in \cite[Proposition~3.9]{s20}, the validity of the SOSC in \eqref{sosc} at $(\ox,\oy)$, a solution to the KKT system \eqref{vs}, implies noncriticality of the Lagrange multiplier $\oy$. However, as demonstrated via an NLP example in \cite[Example~3]{is12}, the SOSC is, in general, strictly stronger than the noncriticality condition.
The core importance of noncritical multipliers is captured in the following result, derived from \cite[Theorem~3.6 \& Proposition~3.8]{s20}. Such an equivalence was initially identified in \cite[Theorem~2.3]{iz08} within the context of NLPs, and subsequently extended to broader classes of constrained and composite optimization problems (see, e.g., \cite{ms18}).

\begin{Theorem}\label{chcri} Assume 
 that $(\ox,\oy)$ is a solution to the KKT system \eqref{vs}. Then the following  conditions are equivalent.
 \begin{itemize}[noitemsep,topsep=0pt]
 \item [{\rm (a)}] The multiplier $\oy$ is noncritical for the KKT system \eqref{vs}.
 \item [{\rm (b)}] There are numbers $\ve>0$ and $\tau\ge 0$ such that the error bound estimate
\begin{equation}\label{subr3.1}
\|x-\ox\|+\dist\big(y,\Lambda(\ox)\big)\le\tau \, r(x,y)
\end{equation}
holds for any $(x,y)\in\B_\ve(\ox,\oy)$, where the Lagrange multiplier set $\Lambda(\ox)$ is defined in \eqref{laset} and the residual function $r$ comes from \eqref{res}. 
 \end{itemize}
\end{Theorem}

We demonstrate below  that  Propositions~\ref{prop:solv} and \ref{prop:eb} can be justified even if we weaken the SOSC in \eqref{sosc}  to the noncriticality of the associated Lagrange multiplier.
These, however, can be achieved under an extra relative interior assumption on the multiplier $\oy$ and $g$ being piecewise linear. 
Note that a CPLQ function $g$ defined as in \eqref{PWLQ} is polyhedral (convex piecewise linear) if $A_i=0$ for all $i\in I$. 

\begin{Proposition}\label{prop:solv2}
Let $(\ox, \oy)$ be a solution to the KKT system \eqref{vs} with $g$ being a polyhedral function. Assume that $\oy\in \ri \sub g(\Phi(\ox))$ is noncritical and $\sigma$ satisfies both error bound estimates in \eqref{eb}.
Then there exists a positive constant $\varepsilon$ such that for all $(x, y)\in \B_\varepsilon(\ox, \oy)$ with $\sigma(x, y)>0$, the solution set to the variational system in \eqref{au-KKT} is always nonempty.
\end{Proposition}

\begin{proof}
Set $\oz:=\Phi(\ox)$ and $K:=K_g(\oz,\oy)$.
Since $g$ is polyhedral, we get from \eqref{domcr} that $\d^2g(\oz,\oy)=\dd_K$, and hence,
\begin{equation*}
    D(\sub g)(\oz,\oy)=\partial\left(\tfrac{1}{2} \d^2g(\oz,\oy)\right)=N_K.
\end{equation*}
Moreover,  the critical cone $K=N_{\sub g(\oz)}(\oy)$ is a linear subspace, due to $\oy\in\ri\,\sub g(\oz)$, and we have
\begin{equation*}
    N_K(v)=K^\perp \qquad \mbox{for all } \;v\in K.
\end{equation*}
Thus, the variational  system \eqref{au-KKT} takes the form
\begin{equation}\label{eq:linear}
\begin{cases}
\nabla\varphi(x)+\nabla^2_{xx}L(x,y)(s-x)+\nabla\Phi(x)^*q=0,\\
\Phi(x)+\nabla\Phi(x)(s-x)-\sigma(x,y)(q-y)-\bar z\in K,\\
q-\bar y\in K^\perp.
\end{cases}
\end{equation}
Set $u:=s-x$ and $w:=q-\oy$.
We can then equivalently express \eqref{eq:linear} as the problem of finding
$(u, w)\in\X\times K^\perp$ such that
\begin{equation}\label{eq:projected}
\begin{cases}
\nabla^2_{xx}L(x,y)u+\nabla\Phi(x)^*w=-\nabla_xL(x,\oy),\\
P_{K^\perp}\big(\nabla\Phi(x)u-\sigma(x,y)w\big)=-P_{K^\bot}\big(\Phi(x)-\oz +\sigma(x, y)(y-\oy)\big).
\end{cases}
\end{equation}
Given the pair $(x,y)\in \R^n\times \R^m$, we define the linear operator $ M_{x,y}:\R^n\times K^\perp\to \R^n\times K^\perp$
by
\begin{equation*}
    M_{x,y}(u,w):=
  \Big(\nabla^2_{xx}L(x,y)u+\nabla\Phi(x)^*w,
  P_{K^\perp}\big(\nabla\Phi(x)u-\sigma(x,y)w\big)\Big).
\end{equation*}
We claim that for any pair $(x,y)$ sufficiently close to $(\ox,\oy)$ with $\sigma(x,y)>0$,  $M_{x,y}$ is nonsingular. This clearly implies that   \eqref{eq:projected} has a solution.
To justify this claim, suppose on the contrary that it fails.
That means  we can find sequences $(\xk,\yk)\to(\ox,\oy)$, with $\sigk:=\sigma(\xk,\yk)>0$ for all $k$, and $(u^k,w^k)\neq (0, 0)$ in $\R^n\times K^\perp$ satisfying
\begin{equation}\label{eq:non}
  H_ku^k+\nabla\Phi(\xk)^*w^k=0,
  \quad
  \nabla\Phi(\xk)u^k- \sigk w^k \in K,
\end{equation}
where $H_k:=\nabla^2_{xx}L(\xk, \yk)$. 
Observe from the inclusion in \eqref{eq:non} that $u^k\neq0$. 
Since, otherwise, we get $\sigk w^k \in K$, and therefore $w^k=0$ due to $\sigk\neq0$ and $w^k\in K^\bot$.
We can then assume, without loss of generality, that the sequence $u^k/\|u^k\|$ converges to some nonzero vector $\xi$.
The relations in \eqref{eq:non} then give us
\begin{equation}\label{eq:non1}
H_k\frac{u^k}{\|u^k\|} + \big(\nabla\Phi(\xk)-\nabla\Phi(\ox)\big)^*\frac{w^k}{\|u^k\|}\in \nabla\Phi(\ox)^*K^\bot,
\quad
\nabla\Phi(\xk)\frac{u^k}{\|u^k\|} - \sigk \frac{w^k}{\|u^k\|} \in K.
\end{equation}
We now claim that $w^k = O(\|u^k\|)$.
Assume otherwise that $w^k\neq 0$ for all $k$ and $u^k = o(\|w^k\|)$ as $k\to\infty$. Dividing the equality and inclusion in \eqref{eq:non} respectively by $\|w^k\|$ and $\sigk\|w^k\|$, we get that
\begin{equation}\label{eq:non2}
H_k\frac{u^k}{\|w^k\|} + \nabla\Phi(\xk)^*\frac{w^k}{\|w^k\|} = 0,
\quad
\frac{\nabla\Phi(\xk)}{\sigk}\frac{u^k}{\|w^k\|} -\frac{w^k}{\|w^k\|} \in K.
\end{equation}
Assume, without loss of generality, that $w^k/\|w^k\| \to \eta \neq 0$. Passing to the limit in the first relationship in \eqref{eq:non2}    yields $\nabla\Phi(\ox)^*\eta = 0$.
Denote $E:= \ker\nabla\Phi(\ox)^*$.
We know that the projection $P_E$ is a linear mapping satisfying $P_E(\nabla\Phi(\ox)\X) = \{0\}$.
Since  $\xk - \ox = O(\sigk)$ due to  \eqref{eb} and $u^k = o(\|w^k\|)$, we have
\begin{equation*}
    P_E\Big(\frac{\nabla\Phi(\xk)}{\sigk}\frac{u^k}{\|w^k\|}\Big) = P_E\Big(\frac{\nabla\Phi(\xk)-\nabla\Phi(\ox)}{\sigk}\frac{u^k}{\|w^k\|}\Big) \to 0\quad\textrm{as }\; k\to\infty,
\end{equation*}
This, combined with the inclusion in \eqref{eq:non2}, gives us $-P_E(\eta) \in P_E(K)$.
We can then select $\hat\eta\in K$ with $P_E(\hat\eta) = P_E(\eta) = \eta$, since $\eta \in E$. On the other hand, it follows from $w^k\in K^\perp$ that $\eta\in K^\perp$. Thus, we obtain 
\begin{equation*}
    0 = \la \eta, \hat\eta\ra =  \la \eta, P_E(\hat\eta)+P_{E^\perp}(\hat\eta)\ra = \la \eta, P_E(\hat\eta)\ra = \|\eta\|^2.
\end{equation*}
We then have $\eta = 0$, a contradiction. This proves our cliam that $w^k = O(\|u^k\|)$. 
Recalling the error bounds \eqref{eb}, we get $\sigk\to0$ as $k\to\infty$.
Passing to the limit in both relationships in \eqref{eq:non1} and taking into account that $\nabla\Phi(\ox)^*K^\bot$ is a convex polyhedral set, we arrive at
\begin{equation*}
  \nabla\Phi(\ox)\xi\in K,\qquad \nabla^2_{xx}L(\ox, \oy)\xi \in \nabla\Phi(\ox)^*K^\bot=\nabla\Phi(\ox)^*D(\sub g)(\oz,\oy)(\nabla\Phi(\ox)\xi).
\end{equation*}
This along with $\xi\neq 0$ tells us  that $\oy$ is a critical multiplier, a contradiction, and therefore completes the proof. 
\end{proof}

The solvability of the sSQP subproblems, derived in Proposition~\ref{prop:solv2}, generalizes a similar observation in \cite[Lemma~2]{is12} for NLPs with equality constraints. Note that the latter was established using different arguments based on the nonsingularity of perturbed matrices obtained in \cite[Lemma~1]{is12}. 
It is worth mentioning that the relative interior condition was not explicitly stated therein; however, it is trivially satisfied since such a framework corresponds to  choosing the modeling function $g = \delta_{\{0\}^m}$, which tells us  $\sub g(z) = \Y$ for all $z$.

One may ask whether the relative interior condition in Proposition~\ref{prop:solv2} can be dropped for the solvability of the sSQP subproblems. The following example illustrates that  the solvability may fail of this condition is not satisfied.

\begin{Example}
    In the framework of the composite optimization problem \eqref{comp}, let $\varphi(x)=-x-\frac12x^2$, $\Phi(x)=x$ for $x\in \R$, and $g(z)=z+\delta_{\R_-}(z)$ for $z\in\R$. 
Set $\ox=0$, $\oy=1$, and $\oz := \Phi(\ox) = 0$. We get $\partial g(0)=[1,+\infty)$ and observe that $\bar y\in\partial g(\oz)$ but $\bar y\notin \ri\partial g(\oz)$.
The Lagrangian, corresponding to this composite problem,  is $ L(x,y)=-x-\frac12x^2+xy$.
Thus, $\nabla_xL(x,y)=-1-x+y$ and $\nabla_xL(0,1)=0$. 
Combining these tells us that $(\ox,\oy)$ is a solution to the associated KKT system in \eqref{vs}. 
On the other hand, the critical cone of $g$ at $\oz$ for $\oy$ is
$K_g(0,1)=N_{\partial g(0)}(1)=\R_-$, which implies that 
$$
    \d^2g(0,1)(w)=\dd_{K_g(0,1)}(w)= \dd_{\R_-}(w)\qquad\mbox{and}\qquad D(\partial g)(0,1)(w)=N_{\R_-}(w)
$$
for any $w\in \R$. Since $\nabla^2_{xx}L(0,1)=-1$ and 
    $\nabla\Phi(0)=1$, the condition \eqref{crc} reads as 
    $$
     0\in -w+N_{\R_-}(w).
    $$
It is not hard to see that  the only solution to this inclusion is $w=0$, and therefore $\bar y=1$ is a noncritical multiplier. 
Observe that the Lagrange multiplier set at $\ox$ is $\Lambda(\ox)=\{1\}$.
Define 
\begin{equation*}
 \sigma(x,y):=|x|+|y-1|=|x|+\dist(y,\Lambda(\ox)), \quad (x, y) \in \R^2.
\end{equation*} 
It is clear that $\sigma$ satisfies the two-sided error bound estimates in \eqref{eb} with constants  $\kappa_1=\kappa_2=1$. 

Set 
$x_t=0$ and  $y_t=1+t$ for $t\in (0,1)$. 
We then have $\sigma(x_t,y_t)=t$.  
We will show that both variational systems \eqref{au-KKT} and \eqref{ssqp-kkt} fail to have a solution with $(x,y)=(x_t,y_t)$ for any $ t\in (0,1)$. The first equation in both variational systems becomes
$0=-1-s+q$, which results in $q = s +1$. 
The second inclusion in the auxiliary system \eqref{au-KKT} reads as
$q-\bar y\in N_{\R_-}\big(s-t(q-y_t)\big)$. 
Substituting $\bar y=1$, $y_t=1+t$, and $q = s + 1$, we obtain
\begin{equation}\label{eq:poly-core}
    s\in N_{\R_-}\big((1-t)s+t^2\big).
\end{equation}
Set $r:=(1-t)s+t^2$. If $r<0$, we deduce from \eqref{eq:poly-core} that $s=0$, which results in $r=t^2>0$, a contradiction. If  $r=0$, then $s=-{t^2}/{(1-t)}<0$. However, \eqref{eq:poly-core} would require that $s\ge 0$, which is not possible. Finally, if $r>0$, then \eqref{eq:poly-core} cannot be fulfilled. 
Thus, the auxiliary system \eqref{au-KKT} has no solution with $(x, y) = (x_t,y_t)$ for any $t\in (0,1)$.

To show that the same conclusion holds for \eqref{ssqp-kkt}, we conclude from $q=s+1$ and $r=(1-t)s+t^2$ that the inclusion in \eqref{ssqp-kkt} becomes 
\begin{equation}\label{sikkt}
q\in \partial g (r).
\end{equation} If $r<0$, we deduce from $\partial g(r)=\{1\}$ that  $s+1=1$. Thus,  $s=0$, which gives $r=t^2>0$, a contradiction.  If $r=0$, we get $s=-{t^2}/{(1-t)}<0$ and therefore $q=s+1<1$.
In this case, it follows from $\partial g(0)=[1,\infty)$ that $q\notin\partial g(0)$, a contradiction with \eqref{sikkt}.  If $r>0$, we clearly have $\partial g(r)=\emptyset$.  Thus, \eqref{ssqp-kkt} has no solution with $(x,y)=(x_t,y_t)$ for any $t\in (0,1)$.
\end{Example}

The next example reveals that the conclusion of Proposition~\ref{prop:solv2} may fail if the modeling function $g$ is CPLQ but is not polyhedral. 
In particular, it shows that the relative interior condition $\oy\in \ri \sub g(\oz)$ alone is insufficient to guarantee the solvability of the sSQP subproblems.

\begin{Example}
 Consider the composite optimization problem \eqref{comp} with $\varphi(x)=-\frac12x^2$, $\Phi(x)=x$ for $x\in \R$, and
  $$
  g(z) = \begin{cases}
    0 & \mbox{if}\;\; z\le 0,\\
    z^2 & \mbox{if}\;\; z> 0.
    \end{cases}
 $$
 Clearly, $g$ is a CPLQ function but is not polyhedral.
 Let $\bar x=\bar y=0$. Then $\bar z:=\Phi(\bar x)=0$ and $\oy\in \sub g(\oz)=\{0\}$.
 The Lagrangian of this problem takes the form $L(x,y)=-\frac12x^2+xy$, and hence $ \nabla_xL(x,y)=-x+y$. Since $\nabla_xL(\ox,\oy)=0$, we conclude that $(\ox,\oy)$ is a KKT point.
 The second subderivative of $g$ at $\oz$ for $\oy$ is given by 
$$
    h(w):=\d^2g(\oz,\oy)(w)=
    \begin{cases}
    0 & \mbox{if}\;\; w\le 0,\\
    2w^2 & \mbox{if}\;\; w>0,
    \end{cases}
$$ 
which in turn yields
$$
D(\partial g)(\oz,\oy)(w)=\partial \Big(\frac{1}{2}\d^2g(\oz,\oy)\Big)(w)= \begin{cases}
    \{0\} & \mbox{if}\;\; w\le 0,\\
    \{2w\} & \mbox{if}\;\; w>0.
    \end{cases}
$$
The inclusion in \eqref{crc} then takes the form 
\[
    0\in \nabla^2_{xx}L(\ox,\oy)w+D(\partial g)(\oz,\oy)(w)
   =\begin{cases}
    \{-w\} & \mbox{if}\;\; w\le 0,\\
    \{w\} & \mbox{if}\;\; w>0.
    \end{cases}
\]
If $w<0$, this gives $0=-w$, which is not possible.    
If $w>0$, it gives $0=w$, again a contradiction. 
The latter inclusion clearly holds for $w = 0$.
Thus the only solution of this inclusion is $w=0$, meaning that $\bar y$ is a noncritical Lagrange multiplier.
Moreover, the Lagrange multiplier set at $\ox$ is given by $\Lambda(\ox)=\{0\}$.
Define 
$$  
\sigma(x,y):=|x|+|y|=|x|+\dist(y,\Lambda(\ox)), \quad (x, y) \in \R^2.
$$
It can be seen that $\sigma$ satisfies the two-sided error bound estimates in \eqref{eb} with constants  $\kappa_1=\kappa_2=1$. 

Set
$x_t=0$ and  $y_t=t$ for $t\in (0,1/2)$. We then get 
 $\sigma(x_t,y_t)=t$.  We will show that both variational systems \eqref{au-KKT} and \eqref{ssqp-kkt} fail to have a solution with $(x,y)=(x_t,y_t)$ for any $t\in (0, 1/2)$. The first equation in both variational systems takes the form
$0=-s+q$, which yields $s=q$. Set $r:=(1-t)s+t^2$. The second inclusion in the auxiliary system \eqref{au-KKT} then reads as
\begin{equation}\label{eq:cplq-core}
    s\in \frac12\partial h(r).
\end{equation}
Suppose that $r\le0$.  Since  $\partial h(r)=\{0\}$, it follows from  \eqref{eq:cplq-core} that $s=0$, which yields $r=t^2>0$, a contradiction.  If $r>0$, we deduce from \eqref{eq:cplq-core} that $s=2r=2((1-t)s+t^2)$. This leads to $r = s/2={t^2}/{(2t-1)}<0$ due to  $t\in (0,1/2)$,
which contradicts $r>0$.  Thus the auxiliary system has no solution with $(x, y) =(x_t,y_t)$ for any $t\in (0,1/2)$. 

To justify the same conclusion for \eqref{ssqp-kkt}, we observe first that 
$$
\partial g(z)=
    \begin{cases}
    \{0\} & \mbox{if}\;\; z\le 0,\\
    \{2z\} & \mbox{if}\;\; z>0.
    \end{cases}
$$
Thus, the inclusion $q\in\partial g(r)$ in \eqref{ssqp-kkt}, together with $q=s$  and $r=(1-t)s+t^2$, leads us to the same cases as those of the auxiliary system,  discussed above.  Therefore, the variational system \eqref{ssqp-kkt} has no solution with $(x,y)=(x_t,y_t)$ whenever $t\in(0,1/2)$.
\end{Example}

We next establish an error bound estimate similar to the one in  Proposition~\ref{prop:eb} when the SOSC is replaced with noncriticality of Lagrange multipliers. 
\begin{Proposition}\label{prop:eb2}
Let $(\ox, \oy)$ be a solution to the KKT system \eqref{vs} with $g$ being a polyhedral function. 
Assume that $\oy\in \ri \sub g(\Phi(\ox))$ is noncritical and $\sigma$ satisfies both error bound estimates in \eqref{eb}.
Then there exist positive constants $\varepsilon$ and $\kappa$ such that for all $(x, y)\in \B_\varepsilon(\ox, \oy)$ with $\sigma(x, y)>0$, the error bound estimate \eqref{eb3} holds for any solution $(s, q)$ to the system \eqref{au-KKT}.
\end{Proposition}

\begin{proof}
Assume by contradiction that the estimate \eqref{eb3} fails. 
Recall the equivalent form \eqref{eq:linear} of \eqref{au-KKT} when $g$ is a polyhedral function and $\oy\in \ri \sub g(\Phi(\ox))$.
Then there exist sequences $(\xk, \yk)\to (\ox, \oy)$, with $\sigma_k:= \sigma(\xk, \yk)>0$, and $(s^k, q^k)$ such that
\begin{equation}\label{s8.1}
\|s^k-\xk\|+\|q^k-\yk\|>k \sigma_k,
\end{equation}
where $(s^k, q^k)$ satisfies
\begin{equation}\label{s9.1}
    0=\nabla\varphi(\xk)+ \Hk(s^k-\xk)+\nabla\Phi(\xk)^*q^k, \quad \zk - \oz \in K, \quad q^k - \oy \in K^\bot,
\end{equation}
where $\Hk := \nabla^2_{xx}L(\xk, \yk)$, $\zk :
= \Phi(\xk)+\nabla\Phi(\xk)(s^k-\xk)-\sigma_k(q^k-\yk)$, $\oz := \Phi(\ox)$, and $K := K_g(\oz, \oy)$.
Set $\beta_k := \|s^k-\xk\|+\|q^k-\yk\|$ and $\hyk := P_{\Lm(\ox)}(\yk)$.
Repeating the arguments leading to \eqref{s12} and \eqref{s16} in the proof of Proposition~\ref{prop:eb}, which do not employ SOSC \eqref{sosc}, we get
\begin{equation}\label{s12.1}
0=\nabla^2_{xx}L(\ox, \oy)u+\nabla\Phi(\ox)^*w,
\quad 
\nabla\Phi(\ox)u\in K
\end{equation}
where $(u, w)\neq 0$ is defined similarly as in \eqref{s11}.
We now claim that $w\in K^\bot$.
Recall first that $\hyk, \oy\in\sub g(\oz)$ and $\hyk \to \oy$ as $k \to\infty$. The reduction relation in \eqref{rl2}, together with \eqref{pwfor2}, then yields $\hyk - \oy \in K_g(\oz, \oy)^* = K^\bot$. Combining this with the last inclusion in \eqref{s9.1} leads us to
\begin{equation*}
    \frac{q^k - \yk}{\beta_k} - \frac{\hyk - \yk}{\beta_k}= \frac{q^k - \oy}{\beta_k} - \frac{\hyk - \oy}{\beta_k}   \in  K^\bot.
\end{equation*}
Passing through the limit as $k\to\infty$ and recalling that $\|\yk - \hyk\| = o(\beta_k)$, as in \eqref{s10}, we arrive at $w \in K^\bot$, which proves our claim.
This, combined with \eqref{s12.1} and noncriticality of $\oy$, implies that $u=0$. 
The first equality in \eqref{s12.1} then implies that $\nabla\Phi(\ox)^*w=0$. 

Arguing similarly to the second half of the proof of Proposition~\ref{prop:eb}, we end up with $w =0$, which is not possible. This  completes the proof.
\end{proof}

We conclude this section by presenting the solvability of the sSQP subproblems, an analogue of Theorem~\ref{thm:existence}, assuming the hypotheses of Proposition~\ref{prop:solv2}.
The estimate, stated in the theorem below, extends the corresponding result   for NLPs with  equality constraints in \cite[Proposition~2]{is12} for the the composite optimization problem in \eqref{comp} with polyhedral modeling functions. 

\begin{Theorem}\label{thm:existence2}
Let $(\ox, \oy)$ be a solution to the KKT system \eqref{vs} with $g$ being a polyhedral function. 
Assume that $\oy\in \ri \sub g(\Phi(\ox))$ is noncritical and $\sigma$ satisfies both error bound estimates in \eqref{eb}. Then, there exist positive constants $\varepsilon$ and $\kappa$ such that for all $(x, y)\in \B_\varepsilon(\ox, \oy)$ with $\sigma(x, y)>0$, there exists a solution $(s, q)$ to the variational system in \eqref{ssqp-kkt}, satisfying the error bound estimate
\begin{equation*}
\| s-x\|+\| q-y\|\leq \kappa\sigma(x, y).
\end{equation*}
\end{Theorem}

\begin{proof}
The proof is analogous to that of Theorem~\ref{thm:existence}, relying on the reduction lemma for polyhedral functions in \eqref{rl2}.
\end{proof}

The solvability and the error bound estimate in Theorem~\ref{thm:existence2} provide a natural extension of the corresponding results for NLPs with equality constraints in \cite[Theorem~1]{is12}, where the relative interior condition $\oy\in \ri \sub g(\Phi(\ox))$ is automatically satisfied. While \cite[Theorem~2]{is12} handles NLPs with inequality constraints indirectly via a slack reformulation to leverage their equality-constrained counterpart, our approach can deal with inequality constraints directly, and therefore allows to bypass the need for such reformulations.

\section{Convergence Results}\label{sec:convs}

The basic outline of the local stabilized SQP method is described in Algorithm~\ref{algo:sSQP} below.

\begin{algorithm}[H]
\caption{stabilized SQP} \label{algo:sSQP}
\begin{algorithmic}[1]
\Ensure{Choose a primal-dual pair $(x^0, y^0) \in \X\times \Y$, a residual function $\sigma: \X\times\Y \to \R_+$ satisfying \eqref{eb},  and an error control parameter $\kappa>0$. Set $k = 0$.}

        \If{$(\xk, \yk)$ is a KKT point of \eqref{comp}}
        \State\Return{$(\xk, \yk)$;}
        
        \Else

        \State{find $(\xkk, \ykk)$ satisfying \eqref{age1} and
        \[
        \|\xkk - \xk\| + \|\ykk - \yk\| \leq \kappa \sigma(\xk, \yk);
        \]}
        
        \State {increase $k \longleftarrow k+1$ an go to Step 1.}
        \EndIf
\end{algorithmic}
\end{algorithm}

\begin{Example}[sSQP for CPLQ penalty functions]\label{ex:pen}
Consider composite problem \eqref{comp} with the modeling function $g: \Y \to \oR$ defined in \eqref{pen}.
The SOSC \eqref{sosc} in the present setting reads as
\begin{equation}\label{sosc_pen}
\begin{cases}
    \big\la\nabla_{xx}^2L(\ox,\oy)u,u\big\ra+2\theta_{{\cal K},Q}\big(\nabla\Phi(\ox)u\big)>0\\
    \mbox{for all }\;u\in\X\setminus\{0\}\;\mbox{with}\;\nabla\Phi(\ox)u\in{\cal K}^*+\rge Q,    
\end{cases}
\end{equation}
where ${\cal K}:=K_Y(\oy,\Phi(\ox)-Q\oy)$ (cf. \cite[formula~(50)]{HMS21}).
We note also that $g^*(w) = \tfrac{1}{2}\la Qw, w\ra +\delta_Y(w)$ and $\sub g^*(w) = Qw + N_Y(w)$.
Employing the latter, we see that the sSQP update scheme \eqref{sSQP-KKT} takes the form: given the current primal-dual pair $(\xk, \yk)$, the next iterate $(\xkk, \ykk)$ solves the affine variational inclusion
\begin{equation}\label{update_pen}
        \begin{bmatrix}
            0\\0
        \end{bmatrix} \in \begin{bmatrix}
            \nabla_xL(\xk, \yk)\\ -\Phi(\xk) + Q\yk
        \end{bmatrix} + \begin{bmatrix}
            \nabla^2_{xx}L(\xk, \yk) &  \nabla \Phi(\xk)^*\\
            -\nabla\Phi(\xk) & Q + \sigk I 
        \end{bmatrix}\begin{bmatrix}
            x - \xk\\ y - \yk
        \end{bmatrix} + \begin{bmatrix}
            0\\ N_Y(y)
        \end{bmatrix},
\end{equation}
and satisfies the local condition $\|\xkk - \xk\| +\|\ykk - \yk\|\leq \kappa\sigk$, where $\sigk:= \sigma(\xk, \yk)$.
\end{Example}

The roadmap to establish the convergence of the primal-dual sequence $\{(\xk, \yk)\}_{k\in \N}$, generated by Algorithm~\ref{algo:sSQP}, was already described in Theorem~\ref{thm:fischer}.
The solvability of sSQP subproblems is established separately in Theorem~\ref{thm:existence} for the general setting of the composite problem \eqref{comp} under the SOSC and in Theorem~\ref{thm:existence2} for the polyhedral setting under noncriticality of the multiplier in question, which is taken from the relative interior of the set of subgradients.
We will show below that verification of the other two assumptions (a) and (c) of Theorem~\ref{thm:fischer} are identical for both cases.
We begin with the calmness of the solution mapping required by assumption (a) in the latter result.
Define   the solution mapping $\s: \X\times \Y \tto \X\times \Y$ to the generalized equation \eqref{geco} with a canonical perturbation by
\begin{equation}\label{solm2}
    \s(u, w) := \big\{(x, y) \in \X\times \Y~\big|~(u, w) \in \Psi(x, y) + G(x, y)\big\},\quad (u, w)\in \X\times \Y,
\end{equation}
where $\Psi$ and $G$ are taken from \eqref{ALMge}. Clearly, $(\ox, \oy) \in \s(0, 0)$ if and only if $(\ox, \oy)$ is a solution to the KKT system \eqref{vs}.
According to \cite[Proposition~3.8]{s20}, a certain calmness property of $\s$ at $(0, 0)$ for $(\ox, \oy)$ is equivalent to the error bound estimate \eqref{subr3.1}. 
In the following, we establish the calmness of $\s$ in assumption (a) in Theorem~\ref{thm:fischer} as a consequence of the two-sided error bound estimate in \eqref{subr3}.

\begin{Proposition}[calmness of $\s$]\label{prop:calm}
Let $(\ox, \oy)$ be a solution to the KKT system \eqref{vs} such that there exists a positive constant $\ve$ for which the error bound estimates in \eqref{subr3} hold for all $(x,y)\in\B_\ve(\ox,\oy)$.
Then the solution mapping $\s$, defined in \eqref{solm2}, is calm with some constant $\ell \in (0, 2/\tau_1]$, where $\tau_1$ is taken from \eqref{subr3}.
\end{Proposition}

\begin{proof}
Observing that $r(x, y) = 0$ for all $(x, y)\in  \s(0, 0)$, we   deduce from the first estimate in \eqref{subr3} that
\begin{equation}\label{calm1}
    \s(0, 0) \cap B_\ve(\ox, \oy) = \big(\{\ox\}\times\Lambda(\ox)\big) \cap B_\ve(\ox, \oy).
\end{equation}
Moreover, as discussed above, we can infer from \cite[Proposition~3.8]{s20} and the first estimate in \eqref{subr3} that there exists $ 0<\ve'\leq \ve$ such that 
\begin{equation}\label{calm2}
    \s(u, w) \cap \B_{\ve'}(\ox, \oy) \subset \big(\{\ox\}\times\Lambda(\ox)\big) + 2\tau_1^{-1}\|(u, w)\|\B
\end{equation}
for all $(u, w) \in \ve'\B$, where the choice of the modulus on the right-hand side is due to the proof of the latter result.
Shrinking $\ve'$ if necessary, we can assume that $\max\{\ve', 2\ve'\tau_1^{-1}\}\leq \ve/2$.
It   follows from \eqref{calm1} and \eqref{calm2} that
\begin{align*}
    \s(u, w)\cap B_{\ve'}(\ox, \oy) &\subset \big(\big(\{\ox\}\times\Lambda(\ox)\big) + 2\tau_1^{-1}\|(u, w)\|\B\big)\cap\B_{\ve/2}(\ox, \oy)\\
    &\subset \big(\{\ox\}\times\Lambda(\ox)\big)\cap \B_\ve(\ox, \oy) + 2\tau_1^{-1}\|(u, w)\|\B\\
    &= \s(0, 0)\cap \B_\ve(\ox, \oy) + 2\tau_1^{-1}\|(u, w)\|\B\\
    &\subset \s(0, 0) + 2\tau_1^{-1}\|(u, w)\|\B,
\end{align*}
for all $(u, w) \in \ve'\B$, which verifies the calmness of $\s$ at $(0, 0)$ for $(\ox, \oy)$ with some constant $\ell \in (0, 2/\tau_1]$.
The proof is then complete.
\end{proof}

We now turn our attention to the last assumption (c) of Theorem~\ref{thm:fischer}, the precision of the approximation $\Psi(\cdot, \cdot)$ by $\A(\cdot, \cdot, \tx, \ty)$, defined in \eqref{A}.

\begin{Proposition}[precision of approximation]\label{prop:precision}
Let $(\ox, \oy)$ be a solution to the KKT system \eqref{vs} such that there exists a positive constant $\ve$ for which the error bound estimates in \eqref{subr3} hold for all $(x,y)\in\B_\ve(\ox,\oy)$. 
Assume also that $\sigma$ satisfies the upper bound in \eqref{eb} for all $(x, y) \in \B_\ve(\ox,\oy)$.
Then, for the mappings $\Psi$ and $\A$ taken from \eqref{ALMge} and \eqref{A}, we have
\begin{equation}\label{pre}
    \|\Psi(x, y) - \A(x, y, \tx, \ty)\| \leq o\big(\dist((\tx, \ty), \s(0, 0))\big)
\end{equation}
for all $(\tx, \ty) \in \B_\ve(\ox, \oy)$ and all $(x, y)$ satisfying 
\begin{equation}\label{pre2}
    \|x-\tx\| + \|y - \ty\|\leq \kappa \sigma(\tx, \ty)
\end{equation}
for some $\kappa>0$.
\end{Proposition}

\begin{proof}
As shown in the proof of Proposition~\ref{prop:calm}, the error bound estimates in \eqref{subr3} imply the local representation of $\s(0, 0)$ in \eqref{calm1}.
This tells us that
\begin{equation*}
    \dist((x, y), \s(0, 0)) = \|x - \ox\| + \dist(y, \Lm(\ox))
\end{equation*}
for all $(x, y)$ close to $(\ox, \oy)$.
By the arguments at the end of Subsection~\ref{sect:fischer} and by the upper bound   in \eqref{eb}, we know that the estimate in \eqref{pre} holds for all $(\tx, \ty)$ sufficiently close to $(\ox, \oy)$ and all $(x, y)$ satisfying \eqref{pre2}.
The proof is then complete.
\end{proof}

Combining Theorem~\ref{thm:fischer} with the discussion above leads us to 
the convergence results for the proposed sSQP method, which will be summarized below.   
\begin{Theorem}
Let $(\ox, \oy)$ be a solution to the KKT system \eqref{vs} such that the SOSC \eqref{sosc} holds at $(\ox, \oy)$ and  $\sigma$ satisfy both error bound estimates in \eqref{eb}. 
If $(x^0, y^0)$ is chosen sufficiently close to $(\ox, \oy)$, then the sSQP iterates $\{(\xk, \yk)\}_{k\in \N}$ generated by Algorithm~{\rm \ref{algo:sSQP}} are well-defined and converge  to $(\ox, \hat y)$ for some $\hat y \in \Lambda(\ox)$. Moreover, the rate of convergence is Q-superlinear.

The same conclusion holds when the SOSC is replaced with the noncriticality of the Lagrange multiplier $\oy$, provided that $\oy\in \ri \sub g(\Phi(\ox))$ and $g$ is a polyhedral function.
\end{Theorem}

\section{Concluding Remarks}
In this paper, we have conducted a local convergence analysis of the sSQP method for a broad class of composite optimization problems. 
Specifically, we proved that superlinear convergence is guaranteed under the SOSC without requiring any constraint qualification, and that the SOSC can be further relaxed to the noncriticality of Lagrange multipliers when the modeling function $g$ in \eqref{comp} is polyhedral and a relative interior condition is satisfied. 
While most of our proofs are easily adaptable to other classes of constrained optimization problems, the proofs of Proposition~\ref{prop:eb} and Theorem~\ref{thm:existence} inherently rely on the reduction lemma for CPLQ functions. 
This raises the interesting question of how to bypass the latter for optimization problems that lack polyhedral structures, such as semidefinite programs. 
Finally, establishing a globalized version of our framework for \eqref{comp}, analogous to existing results for NLPs (e.g., \cite{gkr,gr,yo, isu}), remains an important topic for future research.

\end{document}